\documentclass[reqno,10pt]{amsart}
\usepackage{setspace,tikz,xcolor,mathrsfs,listings,multicol,mathabx}
\usepackage{rotating}
\usepackage[vcentermath]{youngtab}
\usepackage{hyperref}
\usetikzlibrary{arrows,matrix}

\newcommand{\g}{\mathfrak{g}}
\newcommand{\HH}{\mathcal{H}}
\newcommand{\inner}[2]{\langle #1, #2 \rangle}
\newcommand{\iso}{\cong}
\renewcommand{\mid}{:}

\newcommand{\RC}{\operatorname{RC}} % rigged configurations
\newcommand{\hwRC}{\RC^*} % highest weight rigged configurations
\newcommand{\sgn}{\!\operatorname{-sgn}}

\newcommand{\wt}{\mathrm{wt}}
\newcommand{\ZZ}{\mathbf{Z}}

\newcommand{\virtual}[1]{\widehat{#1}}
\newcommand{\folding}{\searrow}

\usepackage{listings}
\lstdefinelanguage{Sage}[]{Python}
{morekeywords={False,sage,True},sensitive=true}
\definecolor{dblackcolor}{rgb}{0.0,0.0,0.0}
\definecolor{dbluecolor}{rgb}{0.01,0.02,0.7}
\definecolor{dgreencolor}{rgb}{0.2,0.4,0.0}
\definecolor{dgraycolor}{rgb}{0.30,0.3,0.30}

\usepackage{xparse}

\makeatletter
\protected\def\specialmergetwolists{%
  \begingroup
  \@ifstar{\def\cnta{1}\@specialmergetwolists}
    {\def\cnta{0}\@specialmergetwolists}%
}
\def\@specialmergetwolists#1#2#3#4{%
  \def\tempa##1##2{%
    \edef##2{%
      \ifnum\cnta=\@ne\else\expandafter\@firstoftwo\fi
      \unexpanded\expandafter{##1}%
    }%
  }%
  \tempa{#2}\tempb\tempa{#3}\tempa
  \def\cnta{0}\def#4{}%
  \foreach \x in \tempb{%
    \xdef\cnta{\the\numexpr\cnta+1}%
    \gdef\cntb{0}%
    \foreach \y in \tempa{%
      \xdef\cntb{\the\numexpr\cntb+1}%
      \ifnum\cntb=\cnta\relax
        \xdef#4{#4\ifx#4\empty\else,\fi\x#1\y}%
        \breakforeach
      \fi
    }%
  }%
  \endgroup
}
\makeatother

\usepackage{xparse}
\DeclareDocumentCommand\rpp{ m m g }{
	\foreach \x [count=\s from 1] in {#1}{
	        {\ifnum\s=1
	                \draw (0,-\s)--(\x,-\s);
	                \fi}
	   \draw (0,-\s-1) to (\x,-\s-1);
	   \foreach \y in {0, ..., \x} {\draw (\y,-\s)--(\y,-\s-1);}
	}
	\specialmergetwolists{/}{#1}{#2}\ziplist
	\foreach \x/\y [count=\yi from 1] in \ziplist{
	    \node[anchor=west,font=\small] at (\x,-\yi - .5) {$\y$};
	}
	\IfValueT {#3}
	{\foreach \z [count=\zi from 1] in {#3} {\node[anchor=east,font=\small] at (0,-\zi - .5) {$\z$};}}
	{}
}

\theoremstyle{plain}
\newtheorem{thm}{Theorem}[section]
\newtheorem{lemma}[thm]{Lemma}
\newtheorem{conj}[thm]{Conjecture}
\newtheorem{prop}[thm]{Proposition}
\newtheorem{cor}[thm]{Corollary}
\theoremstyle{definition}
\newtheorem{dfn}[thm]{Definition}
\newtheorem{ex}[thm]{Example}
\newtheorem{remark}[thm]{Remark}
\newtheorem{prob}[thm]{Open Problem}
\numberwithin{equation}{section}
\numberwithin{figure}{section}
\numberwithin{table}{section}

\begin{document}
\title{A rigged configuration model for $B(\infty)$}
\author{Ben Salisbury}
\address{Department of Mathematics, Central Michigan University, Mt. Pleasant, MI 48859}
\email{ben.salisbury@cmich.edu}
\urladdr{http://people.cst.cmich.edu/salis1bt/}
\author{Travis Scrimshaw}
\address{Department of Mathematics, University of California, Davis, CA 95616}
\email{tscrim@ucdavis.edu}
\urladdr{https://www.math.ucdavis.edu/~scrimsha/}
\thanks{T.S. was partially supported by NSF grant OCI-1147247.}

\begin{abstract}
We describe a combinatorial realization of the crystals $B(\infty)$ and $B(\lambda)$ using rigged configurations in all symmetrizable Kac-Moody types up to certain conditions.  This includes all simply-laced types and all non-simply-laced finite and affine types.
\end{abstract}

\maketitle

\section{Introduction}
Crystal basis theory is an elegant and fruitful subject born out of the theory of quantum groups.  Defined by Kashiwara in the early 1990s \cite{K90,K91}, crystals provide a natural combinatorial framework to study the representations of Kac-Moody algebras (including classical Lie algebras) and their associated quantum groups.  Their applications span many areas of mathematics, including representation theory, algebraic combinatorics, automorphic forms, and mathematical physics, to name a few.

The study of crystal bases has led researchers to develop different combinatorial models for crystals which yield suitable settings to studying a particular aspect of the representation theory of quantum groups.  For example, highest weight crystals (which are combinatorial skeletons of an irreducible highest weight module over a quantum group) can be modeled using generalized Young tableaux \cite{KM94,KN94}, using the Littelmann path model \cite{L94,L95,L95-2,L97}, using alcove paths \cite{LP07,LP08} or alcove walks \cite{R06}, using geometric methods \cite{BG01,GL05,KS97}, and many others.  The choice of using one model over the other usually depends on the underlying question at hand (and/or on the preference of the author).  In concert with the descriptions for the highest weight crystals, there are several known realizations of the (infinite) crystal $B(\infty)$ (which is a combinatorial skeleton for the Verma module with highest weight $0$), both in combinatorial and geometric settings, which have various applications.
Combinatorially describing the crystal $B(\infty)$ in affine types is still a work in progress (see \cite{KS10,KS13} for a generalization of the tableaux model to the affine setting in certain types), so another combinatorial model for $B(\infty)$ in affine types may prove useful.

Our choice of model will be that of \emph{rigged configurations}, which arise naturally as indexing the eigenvalues and eigenvectors of a Hamiltonian of a statistical model~\cite{B31,KKRY86,KRY86}.  On the other hand, these eigenvectors may also be indexed by one-dimensional lattice paths~\cite{B89,HKKOTY99,HKOTT02,NY97,SW99}, which can be interpreted as highest weight vectors in a tensor product of certain crystals.  In recent years, the implied connection between highest weight vectors in tensor products of Kirillov-Reshetikhin crystals and rigged configurations has been worked out~\cite{OSS03,OSS03II,Sakamoto13,S06,SW10}.

As we will show, the rigged configuration model has simple combinatorial rules for describing the structure which work in all finite, affine, and all simply-laced Kac-Moody types. These combinatorial rules are only based on the nodes of the Dynkin diagram and their neighbors. This allows us to easily describe the embeddings of $B(\lambda)$ into $B(\mu)$, where $\lambda_a \leq \mu_a$ for all indices $a$. Moreover, we can easily describe the so-called virtualization of $B(\lambda)$ inside of a highest weight crystal of simply-laced type via a diagram folding.

The purpose of this paper is to extend the crystal structure on highest weight crystals in finite type in terms of rigged configurations \cite{DS06,OSS13,Sakamoto13,S06,SW10} to other types and to a crystal model for $B(\infty)$ in terms of rigged configurations.  In slightly more detail, the crystal $B(\infty)$ is a direct limit of all highest weight crystals, so by relaxing a certain admissibility condition on elements of the highest weight crystal, we may obtain a representative of an element of $B(\infty)$ given by a rigged configuration.  An added perk of describing $B(\infty)$ using rigged configurations is that the description is type-independent.  However our proofs are \emph{almost} type-independent as we can do our proofs uniformly across all simply-laced finite types, but there will be some changes in the extension to non-simply-laced finite types and then, again, when extending outside of finite type.

The organization of this paper goes as follows. Section \ref{sec:background} gives a background on crystals and rigged configurations. 
In Section~\ref{sec:simplyfinite}, we describe the rigged configuration model for $B(\infty)$ for simply-laced finite types. 
In Section~\ref{sec:arbitrary_sl}, we extend our model for arbitrary simply-laced types.
We extend our model to all finite, affine, and certain indefinite (symmetrizable) types in Section~\ref{sec:non-simply-laced}.
We describe how highest weight crystals sit inside our $B(\infty)$ model using rigged configurations in Section~\ref{sec:projection}.

\medskip
{\bf Notational remark.} The notation $\g$ may denote different objects in different sections, but we will make this clear near the beginning of each (sub)section.

\section{Background}\label{sec:background}
\subsection{Crystals}

Let $\g$ be a symmetrizable Kac-Moody algebra with index set $I$, generalized Cartan matrix $A = (A_{ij})_{i,j\in I}$, weight lattice $P$, root lattice $Q$, fundamental weights $\{\Lambda_i \mid i \in I\}$, simple roots $\{\alpha_i \mid i\in I\}$, and simple coroots $\{h_i \mid i\in I\}$.  There is a canonical pairing $\langle\ ,\ \rangle\colon P^\vee \times P \longrightarrow \ZZ$ defined by $\langle h_i, \alpha_j \rangle = A_{ij}$, where $P^{\vee}$ is the dual weight lattice.

An \emph{abstract $U_q(\g)$-crystal} is a nonempty set $B$ together with maps
\[
\wt \colon B \longrightarrow P, \ \ \
\varepsilon_i, \varphi_i\colon B \longrightarrow \ZZ\sqcup\{-\infty\}, \ \ \
e_i, f_i \colon B \longrightarrow B\sqcup\{0\},
\]
subject to the conditions
\begin{enumerate}
\item $\varphi_i(b) = \varepsilon_i(b) + \langle h_i, \wt(b)\rangle$ for all $i \in I$,
\item if $b\in B$ satisfies $e_ib \neq 0$, then
\begin{enumerate}
\item $\varepsilon_i(e_ib) = \varepsilon_i(b) - 1$,
\item $\varphi_i(e_ib) = \varphi_i(b) + 1$,
\item $\wt(e_ib) = \wt(b) + \alpha_i$,
\end{enumerate}
\item if $b\in B$ satisfies $f_ib \neq 0$, then
\begin{enumerate}
\item $\varepsilon_i(f_ib) = \varepsilon_i(b) + 1$,
\item $\varphi_i(f_ib) = \varphi_i(b) - 1$,
\item $\wt(f_ib) = \wt(b) - \alpha_i$,
\end{enumerate}
\item $f_ib = b^{\prime}$ if and only if $b = e_i b^{\prime}$ for $b,b^{\prime} \in B$ and $i\in I$,
\item if $\varphi_i(b) = -\infty$ for $b\in B$, then $e_i b = f_i b =0$.
\end{enumerate}
The operators $e_i$ and $f_i$, for $i \in I$, are referred to as the \emph{Kashiwara raising} and \emph{Kashiwara lowering operators}, respectively. See \cite{HK02,K91} for details.

\begin{ex}
For a dominant integral weight $\lambda$, the crystal basis 
\[
B(\lambda) = \{ f_{i_k}\cdots f_{i_1}u_{\lambda} \mid i_1,\ldots,i_k \in I,\ k\in\ZZ_{\ge0} \} \setminus \{0\}
\] 
of an irreducible, highest weight $U_q(\g)$-module $V(\lambda)$ is an abstract $U_q(\g)$-crystal. The crystal $B(\lambda)$ is characterized by the following properties.
\begin{enumerate}
\item The element $u_\lambda \in B(\lambda)$ is the unique element such that $\wt(u_\lambda) = \lambda$.
\item For all $i \in I$, $e_i u_{\lambda} = 0$.
\item For all $i \in I$, $f_i^{\inner{h_i}{\lambda} + 1} u_{\lambda} = 0$. 
\end{enumerate}
\end{ex}

\begin{ex}
The crystal basis
\[
B(\infty) = \{f_{i_k}\cdots f_{i_1} u_{\infty} \mid i_1,\dots,i_k\in I,\ k\in\ZZ_{\ge0}\}
\] 
of the negative half $U_q^-(\g)$ of the quantum group is an abstract $U_q(\g)$-crystal.  Some important properties of $B(\infty)$ are the following.
\begin{enumerate}
\item The element $u_\infty \in B(\infty)$ is the unique element such that $\wt(u_\infty) = 0$.
\item For all $i\in I$, $e_iu_\infty =0$.
\item For any sequence $(i_1,\dots,i_k)$ from $I$, $f_{i_k}\cdots f_{i_1}u_\infty \neq 0$.
\end{enumerate}
\end{ex}

An abstract $U_q(\g)$-crystal is said to be \emph{upper regular} if, for all $b\in B$,
\[
\varepsilon_i(b) = \max\{ k \in\ZZ_{\ge0} \mid e_i^k b \neq 0 \}.
\]
Similarly, an abstract $U_q(\g)$-crystal is said to be \emph{lower regular} if, for all $b\in B$,
\[
\varphi_i(b) = \max\{ k \in \ZZ_{\ge0} \mid f_i^kb \neq 0\}.
\]
If $B$ is both upper regular and lower regular, then we say $B$ is \emph{regular}.  In this latter case, we may depict the entire $i$-string through $b\in B$ diagrammatically as
\[
e_i^{\varepsilon_i(b)}b \overset{i}{\longrightarrow}
e_i^{\varepsilon_i(b)-1}b \overset{i}{\longrightarrow}
\cdots \overset{i}{\longrightarrow}
e_ib \overset{i}{\longrightarrow}
b \overset{i}{\longrightarrow}
f_ib \overset{i}{\longrightarrow}
\cdots \overset{i}{\longrightarrow}
f_i^{\varphi_i(b)-1}b \overset{i}{\longrightarrow}
f_i^{\varphi_i(b)}b.
\]
Note that $B(\lambda)$ is a regular abstract $U_q(\g)$-crystal, but $B(\infty)$ is only upper regular.

Let $B_1$ and $B_2$ be two abstract $U_q(\g)$-crystals.  A \emph{crystal morphism} $\psi\colon B_1 \longrightarrow B_2$ is a map $B_1\sqcup\{0\} \longrightarrow B_2 \sqcup \{0\}$ such that
\begin{enumerate}
\item $\psi(0) = 0$;
\item if $b \in B_1$ and $\psi(b) \in B_2$, then $\wt(\psi(b)) = \wt(b)$, $\varepsilon_i(\psi(b)) = \varepsilon_i(b)$, and $\varphi_i(\psi(b)) = \varphi_i(b)$;
\item for $b \in B_1$, we have $\psi(e_i b) = e_i \psi(b)$ provided $\psi(e_ib) \neq 0$ and $e_i\psi(b) \neq 0$;
\item for $b\in B_1$, we have $\psi(f_i b) = f_i \psi(b)$ provided $\psi(f_ib) \neq 0$ and $f_i\psi(b) \neq 0$.
\end{enumerate}
A morphism $\psi$ is called \emph{strict} if $\psi$ commutes with $e_i$ and $f_i$ for all $i \in I$.  Moreover, a morphism $\psi\colon B_1 \longrightarrow B_2$ is called an \emph{embedding} if the induced map $B_1 \sqcup\{0\} \longrightarrow B_2 \sqcup \{0\}$ is injective.

We say an abstract $U_q(\g)$-crystal is simply a \emph{$U_q(\g)$-crystal} if it is crystal isomorphic to the crystal basis of a $U_q(\g)$-module.

Again let $B_1$ and $B_2$ be abstract $U_q(\g)$-crystals.  The tensor product $B_2 \otimes B_1$ is defined to be the Cartesian product $B_2\times B_1$ equipped with crystal operations defined by
\begin{align*}
e_i(b_2 \otimes b_1) &= \begin{cases}
e_i b_2 \otimes b_1 & \text{if } \varepsilon_i(b_2) > \varphi_i(b_1), \\
b_2 \otimes e_i b_1 & \text{if } \varepsilon_i(b_2) \le \varphi_i(b_1),
\end{cases} \\
f_i(b_2 \otimes b_1) &= \begin{cases}
f_i b_2 \otimes b_1 & \text{if } \varepsilon_i(b_2) \ge \varphi_i(b_1), \\
b_2 \otimes f_i b_1 & \text{if } \varepsilon_i(b_2) < \varphi_i(b_1),
\end{cases} \\ 
\varepsilon_i(b_2 \otimes b_1) &= \max\big( \varepsilon_i(b_2), \varepsilon_i(b_1) - \inner{h_i}{\wt(b_2)} \bigr), \\
\varphi_i(b_2 \otimes b_1) &= \max\big( \varphi_i(b_1), \varphi_i(b_2) + \inner{h_i}{\wt(b_1)} \bigr), \\
\wt(b_2 \otimes b_1) &= \wt(b_2) + \wt(b_1).
\end{align*}

\begin{remark}
Our convention for tensor products is opposite the convention given by Kashiwara in~\cite{K91}.
\end{remark}

More generally if $B_1, \dotsc, B_t$ are regular crystals, to compute the action of the Kashiwara operators on the tensor product $B = B_t \otimes \cdots \otimes B_2 \otimes B_1$, we use the \emph{signature rule}.  Indeed, for $i \in I$ and $b = b_t \otimes \cdots \otimes b_2 \otimes b_1$ in $B$, write
\[
\underbrace{+\cdots+}_{\varphi_i(b_t)}\
\underbrace{-\cdots-}_{\varepsilon_i(b_t)}\
\cdots\
\underbrace{+\cdots+}_{\varphi_i(b_1)}\
\underbrace{-\cdots-}_{\varepsilon_i(b_1)}\ .
\]
From the above sequence, successively delete any $(-,+)$-pair to obtain a sequence
\[
i\sgn(b) :=
\underbrace{+\cdots+}_{\varphi_i(b)}\
\underbrace{-\cdots-}_{\varepsilon_i(b)}
\ .
\]
Suppose $1 \leq j_-, j_+ \leq t$ are such that $b_{j_-}$ contributes the leftmost $-$ in $i\sgn(b)$ and $b_{j_+}$ contributes the rightmost $+$ in $i\sgn(b)$.  Then 
\begin{align*}
e_i b &= b_t \otimes \cdots \otimes b_{j_-+1} \otimes e_ib_{j_-} \otimes b_{j_--1} \otimes \cdots \otimes b_1, \\
f_i b &= b_t \otimes \cdots \otimes b_{j_++1} \otimes f_ib_{j_+} \otimes b_{j_+-1} \otimes \cdots \otimes b_1.
\end{align*}

Let $\mathscr{C}$ denote the category of abstract $U_q(\g)$-crystals.  In~\cite{K02}, Kashiwara showed that direct limits exist in $\mathscr{C}$.  Indeed, let $\{B_j\}_{j\in J}$ be a directed system of crystals and let $\psi_{k,j}\colon B_j \longrightarrow B_k$, $j\le k$, be a crystal morphism (with $\psi_{j,j}$ being the identity map on $B_j$) such that $\psi_{k,j} \psi_{j,i} = \psi_{k,i}$.  Let $\vec{B} = \varinjlim B_j$ be the direct limit of this system and let $\psi_j \colon B_j \longrightarrow \vec{B}$.  Then $\vec{B}$ has a crystal structure induced from the crystals $\{B_j\}_{j \in J}$.  Indeed, for $\vec{b} \in \vec{B}$ and $i\in I$, define $e_i\vec{b}$ to be $\psi_j(e_i b_j)$ if there exists $b_j \in B_j$ such that $\psi_j(b_j) = \vec{b}$ and $e_i(b_j) \neq 0$.  This definition does not depend on the choice of $b_j$.  If there is no such $b_j$, then set $e_i\vec{b} = 0$.  The definition of $f_i\vec{b}$ is similar.  Moreover, the functions $\wt$, $\varepsilon_i$, and $\varphi_i$ on $B_j$ extend to functions on $\vec{B}$.

\subsection{Rigged configurations}
Let $\g$ be a symmetrizable Kac-Moody algebra with index set $I$.  Set $\HH = I \times \ZZ_{>0}$.  Consider a multiplicity array
\[
L = \big(L_i^{(a)} \in \ZZ_{\ge0} \mid (a,i) \in \HH\big)
\]
and a dominant integral weight $\lambda$ of $\g$.  We call a sequence of partitions $\nu = \{ \nu^{(a)} \mid a \in I \}$ an $(L, \lambda)$-\emph{configuration} if 
\begin{equation}\label{LL-config}
\sum_{(a,i)\in\HH} i m_i^{(a)} \alpha_a = \sum_{(a,i)\in\HH} i L_i^{(a)} \Lambda_a - \lambda,
\end{equation}
where $m_i^{(a)}$ is the number of parts of length $i$ in the partition $\nu^{(a)}$.  The set of all such $(L,\lambda)$-configurations is denoted $C(L,\lambda)$.  To an element $\nu \in C(L,\lambda)$, define the \emph{vacancy number} of $\nu$ to be 
\begin{equation}
\label{eq:vacancy_numbers}
p_i^{(a)} = p_i^{(a)}(\nu) = \sum_{j \geq 0} \min(i,j) L_j^{(a)} - \sum_{(b,j) \in \HH} \frac{A_{ab}}{\gamma_b} \min(\gamma_a i, \gamma_b j) m_j^{(b)},
\end{equation}
where $\{\gamma_a : a \in I \}$ are some set of positive integers.  If $\g$ is of simply-laced type, we take $\gamma_a = 1$ for all $a\in I$.

Recall that a partition is a multiset of integers (typically sorted in decreasing order).  A \emph{rigged partition} is a multiset of pairs of integers $(i, x)$ such that $i > 0$ (typically sorted under decreasing lexicographic order).  Each $(i,x)$ is called a \emph{string}, where $i$ is called the length or size of the string and $x$ is the \emph{label}, \emph{rigging}, or \emph{quantum number} of the string.  Finally, a \emph{rigged configuration} is a pair $(\nu, J)$ where $\nu \in C(L,\lambda)$ and $J = \big( J_i^{(a)} \big)_{(a, i) \in \HH}$ where each $J_i^{(a)}$ the weakly decreasing sequence of riggings of strings of length $i$ in $\nu^{(a)}$. We call a rigged configuration \emph{valid} if every label $x \in J_i^{(a)}$ satisfies the inequality $p_i^{(a)} \geq x$ for all $(a, i) \in \HH$. We say a rigged configuration is \emph{highest weight} if $x \geq 0$ for all labels $x$. Define the \emph{colabel} or \emph{coquantum number} of a string $(i,x)$ to be $p_i^{(a)} - x$.  
For brevity, we will often denote the $a$th part of $(\nu,J)$ by $(\nu,J)^{(a)}$ (as opposed to $(\nu^{(a)},J^{(a)})$).

\begin{ex}\label{ex:runningrig}
Rigged configurations will be depicted as sequences of partitions with parts labeled on the left by the corresponding vacancy number and labeled on the right by the corresponding rigging.  For example, 
\[
\begin{tikzpicture}[scale=.35,anchor=top]
 \rpp{2,2}{-1,-1}{-1,-1}
 \begin{scope}[xshift=6cm]
 \rpp{2,2,1}{1,1,1}{1,1,1}
 \end{scope}
 \begin{scope}[xshift=12cm]
 \rpp{2,2,1,1}{0,-2,0,0}{0,0,0,0}
 \end{scope}
 \begin{scope}[xshift=18cm]
 \rpp{2,1}{0,0}{0,0}
 \end{scope} 
 \begin{scope}[xshift=24cm]
 \rpp{2,1}{0,0}{0,0}
 \end{scope}
\end{tikzpicture}
\]
is a rigged configuration with $\g = D_5$ and $L$ is given by $L_2^{(1)} = L_1^{(2)} = L_1^{(3)} = 1$ with all other $L_i^{(a)} = 0$.
\end{ex}

Denote by $\hwRC(L, \lambda)$ the set of valid highest weight rigged configurations $(\nu, J)$ such that $\nu \in C(L, \lambda)$. In~\cite{S06}, an abstract $U_q(\g)$-crystal structure was given to rigged configurations, which we recall first by defining the Kashiwara operators.

\begin{dfn}
\label{def:rc_crystal_ops}
Let $(\nu, J)$ be a valid rigged configuration. Fix $a\in I$ and let $x$ be the smallest label of $(\nu,J)^{(a)}$.  
\begin{enumerate}
\item If $x \geq 0$, then set $e_a(\nu,J) = 0$. Otherwise, let $\ell$ be the minimal length of all strings in $(\nu,J)^{(a)}$ which have label $x$.  The rigged configuration $e_a(\nu,J)$ is obtained by replacing the string $(\ell, x)$ with the string $(\ell-1, x+1)$ and changing all other labels so that all colabels remain fixed.
\item If $x > 0$, then add the string $(1,-1)$ to $(\nu,J)^{(a)}$.  Otherwise, let $\ell$ be the maximal length of all strings in $(\nu,J)^{(a)}$ which have label $x$.   Replace the string $(\ell, x)$ by the string $(\ell+1, x-1)$ and change all other labels so that all colabels remain fixed.  If the result is a valid rigged configuration, then it is $f_a(\nu, J)$ .  Otherwise $f_a(\nu, J) = 0$.
\end{enumerate}
\end{dfn}

Let $\RC(L, \lambda)$ denote the set generated by $\hwRC(L, \lambda)$ by the Kashiwara operators. For $(\nu,J) \in \RC(L, \lambda)$, if $f_a$ adds a box to a string of length $\ell$ in $(\nu,J)^{(a)}$, then the vacancy numbers in simply-laced type are changed using the formula
\begin{equation}\label{f-vacancy}
p_i^{(b)} = \begin{cases}
 p_i^{(b)} & \text{if } i \le \ell,\\
 p_i^{(b)} - A_{ab} & \text{if } i > \ell.
\end{cases}
\end{equation}
On the other hand, if $e_a$ removes a box from a string of length $\ell$, then the vacancy numbers must be changed using
\begin{equation}\label{e-vacancy}
p_i^{(b)} = \begin{cases}
 p_i^{(b)} & \text{if } i < \ell,\\
 p_i^{(b)} + A_{ab} & \text{if } i \geq \ell.
\end{cases}
\end{equation}

Let $\RC(L)$ be the closure under the Kashiwara operators of the set $\hwRC(L) = \bigcup_{\lambda\in P^+} \hwRC(L, \lambda)$. Lastly, the weight map $\wt \colon \RC(L) \longrightarrow P$ is defined as
\begin{equation}\label{RC_weight}
\wt(\nu,J) = \sum_{(a,i) \in \HH} i\big( L_i^{(a)}\Lambda_a - m_i^{(a)}\alpha_a\big).
\end{equation}

\begin{ex}
Let $(\nu,J)$ be the rigged configuration from Example~\ref{ex:runningrig}.  Then
\[
e_3(\nu,J) = 
\begin{tikzpicture}[scale=.35,baseline=-20]
 \rpp{2,2}{-1,-1}{-1,-1}
 \begin{scope}[xshift=6cm]
 \rpp{2,2,1}{0,0,1}{0,0,1}
 \end{scope}
 \begin{scope}[xshift=12cm]
 \rpp{2,1,1,1}{2,0,0,0}{2,0,0,0}
 \end{scope}
 \begin{scope}[xshift=18cm]
 \rpp{2,1}{-1,0}{-1,0}
 \end{scope} 
 \begin{scope}[xshift=24cm]
 \rpp{2,1}{-1,0}{-1,0}
 \end{scope}
\end{tikzpicture}
\]
and
\[
f_2(\nu,J) = 
\begin{tikzpicture}[scale=.35,baseline=-20]
 \rpp{2,2}{0,0}{0,0}
 \begin{scope}[xshift=6cm]
 \rpp{2,2,1,1}{-1,-1,-1,-1}{-1,-1,-1,-1}
 \end{scope}
 \begin{scope}[xshift=12cm]
 \rpp{2,2,1,1}{1,-1,1,1}{1,1,1,1}
 \end{scope}
 \begin{scope}[xshift=18cm]
 \rpp{2,1}{0,0}{0,0}
 \end{scope} 
 \begin{scope}[xshift=24cm]
 \rpp{2,1}{0,0}{0,0}
 \end{scope}
\end{tikzpicture}.
\]
Also we have
\begin{align*}
\wt\bigl((\nu, J)\bigr) & = 2\Lambda_1 + \Lambda_2 + \Lambda_3 - 4 \alpha_1 - 5 \alpha_2 - 6 \alpha_3 - 3 \alpha_4 - 3 \alpha_5
\\ & = -\Lambda_1 + \Lambda_2,
\\ \wt\bigl(e_3 (\nu, J)\bigr) & = -\Lambda_1 + 2 \Lambda_3 - \Lambda_4 - \Lambda_5 = -\Lambda_1 + \Lambda_2 + \alpha_3,
\\ \wt\bigl(f_2 (\nu, J)\bigr) & = -\Lambda_2 + \Lambda_3 = -\Lambda_1 + \Lambda_2 - \alpha_2,
\end{align*}
\end{ex}

\begin{thm}[{\cite[Thm.\ 3.7]{S06}}]\label{S06-thm}
Let $\g$ be a simply-laced Lie algebra.  For $(\nu, J) \in \hwRC(L, \lambda)$, let $X_{(\nu, J)}$ be the graph generated by $(\nu, J)$ and $e_a, f_a$ for $a \in I$.  Then $X_{(\nu, J)}$ is isomorphic to the crystal graph $B(\lambda)$ as $U_q(\g)$-crystals.
\end{thm}

\begin{remark}
In~\cite{S06}, elements of $X_{(\nu,J)}$ were called unrestricted rigged configurations and the graph $X_{(\nu,J)}$ was denoted $X_{(\overline{\nu},\overline{J})}$.
\end{remark}

We note that our condition for highest weight rigged configurations is equivalent to the rigged configuration being highest weight in the sense of a crystal of type $\g$; i.e., that the action of all $e_a$ on a highest weight rigged configuration is $0$.

In the sequel, set $\nu_\emptyset$ to be the multipartition with all parts empty; that is, set $\nu_{\emptyset} = (\nu^{(1)}, \ldots, \nu^{(n)})$ where $\nu^{(a)}_i = \emptyset$ for all $(a,i) \in \HH$.  Therefore the rigging $J_{\emptyset}$ of $\nu_\emptyset$ must be $J_i^{(a)} = \emptyset$ for all $(a,i) \in \HH$.  
When discussing the highest weight crystals $X_{(\nu_{\emptyset}, J_{\emptyset})}$, we will choose our multiplicity array $L$ to be such that 
\begin{displaymath}
\sum_{(a,i)\in \HH} iL_i^{(a)}\Lambda_a = \lambda.
\end{displaymath}
It is clear that there are several choices of $L$ that may fit this condition, but this does not affect the crystal structure.  

\begin{dfn}\label{RClambda_def}
Define $\RC(\lambda)$ to be $X_{(\nu_\emptyset,J_\emptyset)}$ for any symmetrizable Kac-Moody algebra.
\end{dfn}

\section{Rigged configuration model for $B(\infty)$ in simply-laced finite type}
\label{sec:simplyfinite}

For this section, unless otherwise noted, let $\g$ be a Lie algebra of simply-laced finite type.  We wish to generate a model for $B(\infty)$ with $(\nu_\emptyset,J_\emptyset)$ as its highest weight vector.  By choosing a fixed $\lambda > 0$, for any $(\nu,J) \in \RC(\lambda)$, there exists $k\ge 0$ such that $f_a^k(\nu,J) = 0$ by the validity condition given in Definition~\ref{def:rc_crystal_ops}.  Therefore, we need a modified Kashiwara operator $f_a^{\prime}$ (for $a \in I)$ on rigged configurations to allow the condition $(f_a^{\prime})^k(\nu,J) \neq 0$ for all $k \geq 0$.  To do so, simply define $f_a^{\prime}$ by the same process given in Definition \ref{def:rc_crystal_ops} with the validity condition omitted and choose $\lambda = 0$.  

\begin{dfn}
For any symmetrizable Kac-Moody algebra $\g$ with index set $I$, define $\RC(\infty)$ to be the graph generated by $(\nu_\emptyset,J_\emptyset)$, $e_a$, and $f_a^{\prime}$, for $a \in I$, where $e_a$ acts on elements $(\nu,J)$ in $\RC(\infty)$ using the same procedure as in Definition~\ref{def:rc_crystal_ops}.  
\end{dfn}

The remainder of the crystal structure is given by
\begin{subequations}\label{epphiwt}
\begin{align}
\label{Binf_ep} \varepsilon_a(\nu,J) &= \max\{ k \in \ZZ_{\ge0} \mid  e_a^k(\nu,J) \neq 0 \}, \\ 
\label{Binf_phi} \varphi_a(\nu,J) &= \varepsilon_a(\nu,J) + \langle h_a,\wt(\nu,J)\rangle, \\ 
\label{Binf_wt} \wt(\nu,J) &= -\sum_{(a,i)\in\HH} im_i^{(a)}\alpha_a = -\sum_{a\in I} |\nu^{(a)}|\alpha_a.
\end{align}
\end{subequations}
It is worth noting that, in this case, the definition of the vacancy numbers reduces to
\begin{equation}
p_i^{(a)}(\nu) = p_i^{(a)} = -\sum_{(b,j) \in \HH} A_{ab}\min(i, j) m_j^{(b)}.
\end{equation}

\begin{ex}
Let $\g$ be of type $A_5$ and $(\nu,J)$ be the rigged configuration
\begin{align*}
(\nu,J) &= 
\begin{tikzpicture}[scale=.35,baseline=-18]
 \rpp{1}{-1}{-1}
\begin{scope}[xshift=5cm]
 \rpp{2}{-1}{-2}
\end{scope}
\begin{scope}[xshift=12cm]
 \rpp{1}{1}{0}
\end{scope}
\begin{scope}[xshift=16cm]
 \rpp{1}{-1}{0}
\end{scope}
\begin{scope}[xshift=21cm]
 \rpp{2}{-1}{-3}
\end{scope}
\end{tikzpicture}
\intertext{Then $\wt(\nu,J) = -\alpha_1 - 2\alpha_2 - \alpha_3 - \alpha_4 - 2\alpha_5$,}
e_2(\nu,J) &= 
\begin{tikzpicture}[scale=.35,baseline=-18]
 \rpp{1}{-1}{-1}
\begin{scope}[xshift=5cm]
 \rpp{1}{0}{0}
\end{scope}
\begin{scope}[xshift=12cm]
 \rpp{1}{1}{0}
\end{scope}
\begin{scope}[xshift=16cm]
 \rpp{1}{-1}{0}
\end{scope}
\begin{scope}[xshift=21cm]
 \rpp{2}{-1}{-3}
\end{scope}
\end{tikzpicture}
\intertext{and}
f_2(\nu,J) &= 
\begin{tikzpicture}[scale=.35,baseline=-18]
 \rpp{1}{-1}{-1}
\begin{scope}[xshift=5cm]
 \rpp{3}{-2}{-4}
\end{scope}
\begin{scope}[xshift=12cm]
 \rpp{1}{1}{0}
\end{scope}
\begin{scope}[xshift=16cm]
 \rpp{1}{-1}{0}
\end{scope}
\begin{scope}[xshift=21cm]
 \rpp{2}{-1}{-3}
\end{scope}
\end{tikzpicture}
\end{align*}
\end{ex}

\begin{lemma}\label{RCcrystal}
The set $\RC(\infty)$ is an abstract $U_q(\g)$-crystal with Kashiwara operators $e_a$ and $f_a^{\prime}$ and remaining crystal structure given in equation~\eqref{epphiwt}.
\end{lemma}

\begin{proof}
This proof here is similar to that given in Proposition~9 of~\cite{Sakamoto13}.  We need to show the following, for $(\nu,J)$ in $\RC(\infty)$.
\begin{enumerate}
\item\label{RCinf-einv} If $e_a(\nu,J) \neq 0$ for $a\in I$, then $f_a^{\prime} e_a(\nu,J) = (\nu,J)$.
\item\label{RCinf-finv} For any $a\in I$, we have $e_a f_a^{\prime}(\nu,J) = (\nu,J)$.
\item\label{RCinf-ewt} If $e_a(\nu,J) \neq 0$ for $a \in I$, then $\wt\bigl(e_a(\nu,J)\bigr) = \wt(\nu,J) + \alpha_a$.
\item\label{RCinf-fwt} For $a\in I$, $\wt\bigl(f_a'(\nu,J)\bigr) = \wt(\nu,J) - \alpha_a$.
\item\label{RCinf-epphi} For $a\in I$, $\varepsilon_a\bigl(f_a^{\prime} (\nu,J)\bigr) = \varepsilon_a(\nu,J) + 1$ and $\varphi_a\bigl(f_a^{\prime} (\nu,J)\bigr) = \varphi_a(\nu,J) - 1$.
\end{enumerate}
Let $(\nu,J)$ be an arbitrary rigged configuration in $\RC(\infty)$.  In what follows, we will suppose that $m_i^{(a)}$ is the number of parts of length $i$ in the partition $\nu^{(a)}$ and that $x$ is the smallest label of $(\nu,J)^{(a)}$.  Set $(\nu',J') = f_a'(\nu,J)$ and $(\nu'',J'') = e_a(\nu,J)$. 

To prove~(\ref{RCinf-einv}), suppose that $(\nu'',J'')$ is obtained from $(\nu,J)$ by changing the string $(\ell,x)$ to $(\ell-1,x+1)$, so that $\ell$ is the minimal length string among all strings with label $x$.  If $i < \ell$ and $(i,y)$ is a string in $(\nu,J)$, then $p_i^{(a)}(\nu'') = p_i^{(a)}(\nu)$ by~\eqref{e-vacancy}.  Thus $(i,y)$ is unaffected by the action of $e_a$, and $y \ge x+1$.  On the other hand, if $i \ge \ell$, then $p_i^{(a)}(\nu'') = p_i^{(a)}(\nu) + 2$ by~\eqref{e-vacancy}.  Thus $(i,y)$ is replaced by the string $(i,y+2)$ under the action of $e_a$ and $y + 2 > x + 1$.  In both cases, $(\ell-1,x+1)$ is the string with minimal label and longest length, so $f_a'$ will change $(\ell-1,x+1)$ to $(\ell,x)$ and $f_a^{\prime} e_a(\nu,J) = (\nu,J)$, as required.

Suppose that $(\nu',J')$ is obtained from $(\nu,J)$ by changing the string $(\ell,x)$ to $(\ell+1,x-1)$, so $\ell$ is the maximal length of all strings with label $x$.  If $i\le \ell$ and $(i,y)$ is a string in $(\nu,J)^{(a)}$, then $p_i^{(a)}(\nu') = p_i^{(a)}(\nu)$ by~\eqref{f-vacancy}. Thus $(i,y)$ is left unaffected by the action of $f_a'$, and $y > x-1$ because $x$ is the smallest label of $(\nu,J)$.  On the other hand, if $i > \ell$, then $p_i^{(a)}(\nu') = p_i^{(a)}(\nu) - 2$ by~\eqref{f-vacancy}.  Thus $(i,y)$ is replaced by $(i,y-2)$ by the action of $f_a'$ and $y-2 \ge x-1$.  In both cases, $(\ell+1,x-1)$ is the string with minimal label and shortest length, so $e_a$ will change $(\ell+1,x-1)$ to $(\ell,x)$ and $e_af_a'(\nu,J) = (\nu,J)$ to prove~(\ref{RCinf-finv}).

For~(\ref{RCinf-ewt}), if $(\nu'',J'') \neq 0$ for some $a\in I$, then $(\nu'',J'')$ is obtained from $(\nu,J)$ by replacing the string $(\ell,x)$ with $(\ell-1,x+1)$, where $\ell$ is the minimal length of all strings in $(\nu,J)^{(a)}$ having label $x$.  Then $|\nu''^{(a)}| = |\nu^{(a)}|-1$ and the result follows.

To see~(\ref{RCinf-fwt}), if $x > 0$, then the string $(1,-1)$ is added to $(\nu,J)^{(a)}$.  Then $|\nu'^{(a)}| = |\nu^{(a)}| + 1$.  On the other hand, if $x \le 0$ and $\ell$ is the maximal length of all strings in $(\nu,J)^{(a)}$ with label $x$, then the string $(\ell,x)$ is replaced by the string $(\ell+1,x-1)$, so $|\nu'^{(a)}| = |\nu^{(a)}| + 1$.  In both cases, the equality $|\nu'^{(a)}| = |\nu^{(a)}|+1$ yields the desired result.

The first part of~(\ref{RCinf-epphi}) follows immediately from the definition.  To see $\varphi_a\bigl(f_a^{\prime} (\nu,J)\bigr) = \varphi_a\bigl((\nu,J)\bigr) - 1$, we have
\begin{align*}
\varphi_a\bigl(f_a^{\prime} (\nu,J)\bigr) & = \inner{h_a}{\wt\bigl(f_a^{\prime} (\nu,J)\bigr)} + \varepsilon_a\bigl(f_a^{\prime} (\nu,J)\bigr)
\\ & = \inner{h_a}{\wt(\nu,J)} - \inner{h_a}{\alpha_a} + \varepsilon_a(\nu,J) + 1
\\ & = \inner{h_a}{\wt(\nu,J)} - 2 + \varepsilon_a(\nu,J) + 1 
\\ & = \varphi_a(\nu,J) - 1. \qedhere
\end{align*}
\end{proof}

\begin{dfn}
\label{def:T_crystal}
For a weight $\lambda$, let $T_\lambda = \{t_\lambda\}$ be the abstract $U_q(\g)$-crystal with operations defined by
\[
e_at_\lambda = f_at_\lambda = 0 , \ \ \ \
\varepsilon_a(t_\lambda) = \varphi_a(t_\lambda) = -\infty, \ \ \ \ 
\wt(t_\lambda) = \lambda.
\]
\end{dfn}

For any abstract $U_q(\g)$-crystal $B$, the tensor product $T_{\lambda} \otimes B$ has the same crystal graph as $B$, but with each weight shifted by $\lambda$ (and appropriate modifications to $\varepsilon_a$ and $\varphi_a$).  Following~\cite{K02}, there is an embedding of crystals
\[
I_{\lambda+\mu,\lambda}\colon  T_{-\lambda}\otimes B(\lambda) \lhook\joinrel\longrightarrow T_{-\lambda-\mu} \otimes B(\lambda+\mu)
\]
which sends $t_{-\lambda}\otimes u_\lambda  \mapsto t_{-\lambda-\mu}\otimes u_{\lambda+\mu}$ and commutes with $e_a$ for each $a\in I$.  Moreover, for any three dominant weights $\lambda$, $\mu$, and $\xi$, we get a commutative diagram
\begin{equation}
\label{eq:directed_system}
\begin{tikzpicture}[xscale=5,yscale=1.5,baseline=0.6cm]
\node (1) at (0,1) {$T_{-\lambda} \otimes B(\lambda)$};
\node (2) at (1,1) {$T_{-\lambda-\mu}\otimes B(\lambda+\mu)$};
\node (3) at (1,0) {$T_{-\lambda-\mu-\xi}\otimes B(\lambda+\mu+\xi)$.};
\path[->,font=\scriptsize]
 (1) edge node[above]{$I_{\lambda+\mu,\lambda}$} (2)
 (1) edge node[below]{$I_{\lambda+\mu+\xi,\lambda}$} (3)
 (2) edge node[right]{$I_{\lambda+\mu+\xi,\lambda+\mu}$} (3);
\end{tikzpicture}
\end{equation}
Using the order on dominant integral weights given by $\mu\le\lambda$ if and only if $\lambda-\mu \in P^+$, the set $\{T_{-\lambda}\otimes B(\lambda)\}_{\lambda\in P^+}$ is a directed system.  

\begin{thm}[\cite{K02}]
We have $\displaystyle B(\infty) = \varinjlim_{\lambda\in P^+} T_{-\lambda}\otimes B(\lambda)$.
\end{thm}

By Theorem~\ref{S06-thm}, each $B(\lambda)$ is $U_q(\g)$-crystal isomorphic to the graph $\RC(\lambda)$ generated by a highest weight rigged configuration $(\nu,J)$ of weight $\lambda$ in $\RC(L)$ and the Kashiwara operators $e_a$ and $f_a$ defined in Definition \ref{def:rc_crystal_ops}.  Thus we have 
\[
\varinjlim_{\lambda\in P^+} T_{-\lambda} \otimes B(\lambda) \cong \varinjlim_{\lambda\in P^+} T_{-\lambda} \otimes \RC(\lambda).
\]
Our goal is to complete the diagram
\begin{equation}\label{B-RC}
\begin{tikzpicture}[xscale=5,yscale=1.5,baseline=0.6cm]
\node (1) at (0,0) {$B(\infty)$};
\node (2) at (0,1) {$\displaystyle\varinjlim_{\lambda\in P^+} T_{-\lambda} \otimes B(\lambda)$};
\node (3) at (1,1) {$\displaystyle\varinjlim_{\lambda\in P^+} T_{-\lambda} \otimes \RC(\lambda)$};
\node (4) at (1,0) {$\RC(\infty)$};
\path[->] (2) edge node[above]{$\cong$} (3);
\path[->,dashed] (1) edge (4);
\draw[-,double distance=1.5pt] (1) -- (2);
\draw[-,dashed,double distance=1.5pt] (3) -- (4);
\end{tikzpicture}
\end{equation}
by proving that the dashed equality on the right side of the square is actually an equality among $U_q(\g)$-crystals.  Then we may define an isomorphism along the bottom of the square by taking the composite map along the top of the diagram.

\begin{lemma}\label{lemma:embed}
Let $\lambda$ and $\mu$ be dominant integral weights, and let
\[
\widetilde I_{\lambda+\mu,\lambda}\colon T_{-\lambda}\otimes \RC(\lambda) \longrightarrow T_{-\lambda-\mu} \otimes \RC(\lambda + \mu)
\]
by $t_{-\lambda}\otimes(\nu,J) \mapsto t_{-\lambda-\mu} \otimes (\nu', J')$, where $(\nu', J') = (\nu,J)$ as rigged configurations but has vacancy numbers considered as an element of $\RC(\lambda+\mu)$.  For $(\nu,J) \in \RC(\lambda)$, the image $(\nu', J')$ is valid in $\RC(\lambda+\mu)$.  Moreover, $\widetilde I_{\lambda+\mu,\lambda}$ is a crystal embedding.
\end{lemma}

\begin{proof}
Write $\lambda = \sum_{(a,i)\in\HH} iL_i^{(a)}\Lambda_a$ and $\mu = \sum_{(a,i)\in\HH} iK_i^{(a)}\Lambda_a$.  Then 
\begin{align*}
p_i^{(a)}(\nu) &= \sum_{j \geq 1} \min(i,j) L_j^{(a)} - \sum_{(b,j) \in \HH} (\alpha_a | \alpha_b) \min(i,j) m_j^{(b)}\\
&\le \sum_{j\ge 1} \min(i,j)\bigl(L_j^{(a)}+K_j^{(a)}\bigr) - \sum_{(b,j)\in\HH} (\alpha_a|\alpha_b) \min(i,j) m_j^{(b)}\\
&= p_i^{(a)}(\nu').
\end{align*}
Thus
\[
\max J_i^{(a)} = \max J^{\prime(a)}_i \le p_i^{(a)}(\nu) \le p_i^{(a)}(\nu'),
\]
for all $(a, i) \in \HH$ such that $J_i^{(a)} \neq \emptyset$ (and hence $J_i^{\prime(a)} \neq \emptyset$). This proves that $(\nu', J')$ is valid so that $\widetilde I_{\lambda+\mu,\lambda}$ is well-defined.  Moreover, we have
\begin{align*}
\wt\bigl(t_{-\lambda-\mu}&\otimes(\nu', J')\bigr)\\
&= -(\lambda+\mu) + \wt(\nu',J') \\
&= -\sum_{(a,i)\in\HH} i\bigl(L_i^{(a)}+K_i^{(a)}\bigr)\Lambda_a + \sum_{(a,i) \in \HH} i\big( (L_i^{(a)} + K_i^{(a)}) \Lambda_a - m_i^{(a)}\alpha_a) \\
&= -\sum_{(a,i) \in \HH} iL_i^{(a)}\Lambda_a + \sum_{(a,i) \in \HH} i\big( L_i^{(a)}\Lambda_a - m_i^{(a)}\alpha_a\big)\\
&= -\lambda + \wt(\nu,J) \\
&= \wt\bigl(t_{-\lambda}\otimes \wt(\nu,J)\bigr),
\end{align*}
which shows that $\widetilde I_{\lambda+\mu,\lambda}$ preserves the weight map.  Since $\widetilde I_{\lambda+\mu,\lambda}$ is the identity on rigged configurations, we obtain that $e_a$ commutes with $\widetilde I_{\lambda+\mu,\lambda}$ and $\widetilde I_{\lambda+\mu,\lambda}$ preserves $\varepsilon_a$, for all $a\in I$.  Also, $f_a$ commutes with $\widetilde I_{\lambda+\mu,\lambda}$ if $f_a(\nu,J) \neq 0$ because the map is the identity map on rigged configurations.  Then
\begin{align*}
\varphi_a(\nu', J') &= \varepsilon_a(\nu', J') + \langle h_a,\mu+\lambda\rangle \\
&= \varepsilon_a(\nu,J) + \langle h_a,\lambda \rangle + \langle h_a,\mu \rangle \\
&= \varphi_a(\nu,J) + \langle h_a,\mu\rangle,
\end{align*}
so we have
\begin{align*}
\varphi_a\bigl( t_{-\lambda} \otimes (\nu,J) \bigr) &= \max\{ -\infty, \varphi_a(\nu,J) + \langle h_a,-\lambda \rangle \} \\
&= \varphi_a(\nu,J) + \langle h_a,-\lambda \rangle \\
&= \varphi_a( \nu', J') + \langle h_a, -\lambda-\mu\rangle\\
&= \max\{ -\infty, \varphi_a(\nu',J') + \langle h_a, -\lambda-\mu \rangle \} \\
&= \varphi_a\bigl( t_{-\lambda-\mu} \otimes (\nu',J')\bigr) .
\end{align*}
Hence $\widetilde I_{\lambda+\mu,\lambda}$ is a crystal embedding.
\end{proof}

To complete the construction of a directed system of crystals of rigged configurations, we have the following lemma, which follows from a modification of the proof of Lemma \ref{lemma:embed}.

\begin{lemma}\label{lemma:diagram}
For dominant integral weights $\lambda$, $\mu$, and $\xi$, the diagram
\begin{displaymath}
\begin{tikzpicture}[xscale=5,yscale=1.5]
\node (1) at (0,1) {$T_{-\lambda} \otimes \RC(\lambda)$};
\node (2) at (1,1) {$T_{-\lambda-\mu}\otimes \RC(\lambda+\mu)$};
\node (3) at (1,0) {$T_{-\lambda-\mu-\xi}\otimes \RC(\lambda+\mu+\xi)$.};
\path[->,font=\scriptsize]
 (1) edge node[above]{$\widetilde I_{\lambda+\mu,\lambda}$} (2)
 (1) edge node[below]{$\widetilde I_{\lambda+\mu+\xi,\lambda}$} (3)
 (2) edge node[right]{$\widetilde I_{\lambda+\mu+\xi,\lambda+\mu}$} (3);
\end{tikzpicture}
\end{displaymath}
commutes.
\end{lemma}

\begin{proof}
Follows by repeated use of Lemma \ref{lemma:embed} and the fact that $\widetilde I_{-,-}$ is the identity on rigged configurations.
\end{proof}

\begin{lemma}\label{lemma:direct_limit}
We have $\displaystyle \RC(\infty) = \varinjlim_{\lambda\in P^+} T_{-\lambda}\otimes \RC(\lambda)$ as abstract $U_q(\g)$-crystals.
\end{lemma}

\begin{proof}
By Lemmas \ref{lemma:embed} and \ref{lemma:diagram}, $\{T_{-\lambda} \otimes \RC(\lambda) \} _{\lambda \in P^+}$ forms a directed system.  Let $X$ denote the direct limit $\varinjlim T_{-\lambda}\otimes \RC(\lambda)$.  Let $\Theta\colon X \longrightarrow \RC(\infty)$ be the identity map on rigged configurations; that is, for $x \in X$ such that $x = \widetilde{I}_{\lambda}(t_{-\lambda}\otimes (\nu,J))$, we have $\Theta(x) = (\nu,J)$.  To make the setting clear, we will denote the Kashiwara operators on $X$ by $\vec e_a$, $\vec f_a$, the Kashiwara operators on $\RC(\lambda)$ and $T_{-\lambda}\otimes \RC(\lambda)$ by $e_a^\lambda$, $f_a^\lambda$, and the Kashiwara operators on $\RC(\infty)$ by $e_a$, $f_a'$. 

To see that $\Theta$ commutes with Kashiwara lowering operators, for $x\in X$ and $\lambda$ such that $x = \widetilde I_{\lambda}\bigl( t_{-\lambda}\otimes (\nu,J) \bigr)$, we have
\begin{displaymath}
\vec f_ax = \widetilde I_{\lambda}\bigl( f_a^\lambda(t_{-\lambda}\otimes(\nu,J))\bigr)
= \widetilde I_{\lambda}\bigl( t_{-\lambda} \otimes f_a^\lambda(\nu,J)\bigr),
\end{displaymath}
where $t_{-\lambda}\otimes (\nu,J)$ satisfies the condition $f_a^\lambda(\nu,J) \neq 0$.  Note that any such $\lambda$ will suffice by the definition of the direct limit.  Thus
\begin{displaymath}
\Theta(\vec f_ax) = f_a^\lambda(\nu,J) = f_a'(\nu,J) = f_a'\Theta(x).
\end{displaymath}
The calculation involving the Kashiwara raising operators is similar.  By the definition of the weight function, it is clear that $\Theta$ preserves weights.  Moreover, $\Theta$ sends the highest weight vector of $X$ to the highest weight vector $(\nu_\emptyset,J_\emptyset)$ of $\RC(\infty)$, so $\Theta$ is a bijection.  
\end{proof}

\begin{thm}\label{thm:RCinf}
Let $\g$ be a Lie algebra of simply-laced finite type.  Then there exists a $U_q(\g)$-crystal isomorphism $B(\infty) \cong \RC(\infty)$ which sends $u_\infty \mapsto (\nu_\emptyset,J_\emptyset)$.
\end{thm}

\begin{proof}
By Lemma~\ref{lemma:direct_limit}, the dashed arrow on the right-hand of the square in~\eqref{B-RC} becomes an isomorphism of $U_q(\g)$-crystals, so we may construct an isomorphism by composing the maps along the outside of the square.
\end{proof}

\begin{remark}
From this point forward, we denote $f_a^{\prime}$ simply by $f_a$.  This should not cause any confusion.
\end{remark}

\section{Extending Theorem~\ref{thm:RCinf} to arbitrary simply-laced Kac-Moody algebras}
\label{sec:arbitrary_sl}

We show the convexity condition holds for general symmetrizable types.

\begin{lemma}
\label{lemma:convexity}
Consider a rigged configuration $(\nu, J)$. Fix $(a, i) \in \HH$ and suppose that $m_i^{(a)} = 0$. Let $C_{a,b}, C_{a,b}^{\prime}, C_{a,b}^{\vee} \in \ZZ_{>0}$ for all $a,b \in I$, and consider the generalization of the vacancy numbers for $(\nu,J)$ to 
\[
p_i^{(a)} = \sum_{j \geq 1} \min(i,j) L_j^{(a)} - \sum_{(b,j) \in \HH} C_{a,b}^{\vee} A_{ab} \min(C_{a,b} i, C_{a,b}^{\prime} j) m_j^{(b)}
\]
We have
\[
2p_i^{(a)} \geq p_{i-1}^{(a)} + p_{i+1}^{(a)}.
\]
\end{lemma}

\begin{proof}
Consider any $(b,j) \in \HH$ and define
\[
Q_j^{(b)} = \sum_{k=1}^{\infty} \min(C_{a,b} j, C^{\prime}_{a,b} k) m_k^{(b)}.
\]
This is the number of boxes in the first $C_{a,b} j$ columns in the shape $C^{\prime}_{a,b} \nu^{(b)}$.  Set $\Theta_j^{(b)}=Q_j^{(b)} - Q_{j-1}^{(b)}$ and $\Xi_j^{(b)} = Q_{j+1}^{(b)} - Q_j^{(b)}$.  We must have $\Theta_j^{(b)} \geq \Xi_j^{(b)} \geq 0$ since $C^{\prime}_{a,b} \nu^{(b)}$ is a partition. Thus
\begin{align*}
2Q_j^{(b)} &= 2Q_{j-1}^{(b)} + 2 \Theta_j^{(b)} \\
&\geq 2Q_{j-1}^{(b)} + \Theta_j^{(b)} + \Xi_j^{(b)} \\
&= Q_{j-1}^{(b)} + Q_j^{(b)} + \Xi_j^{(b)} \\
&= Q_{j-1}^{(b)} + Q_{j+1}^{(b)}.
\end{align*}
Since $m_i^{(a)} = 0$, we have $\Xi_i^{(a)} = \Theta_i^{(a)}$, and so
\[
2Q_i^{(a)} = Q_{i-1}^{(a)} + Q_{i+1}^{(a)}.
\]
Recall $A_{ab} \leq 0$ for all $a \neq b$ and $C_{a,b}^{\vee} > 0$. Therefore we have $-C_{a,b}^{\vee} A_{ab} \geq 0$ for all $a \neq b$, and hence
\[
-2 \sum_{(b,j) \in \HH} C_{a,b}^{\vee} A_{ab} Q_j^{(b)} \geq - \sum_{(b,j) \in \HH} C_{a,b}^{\vee} A_{ab} \bigl( Q_{j-1}^{(b)} + Q_{j+1}^{(b)} \bigr)
\]
Similarly we can show that
\[
\sum_{j > 0} \min(i,j) L_j^{(a)} \geq \sum_{j > 0} \min(i-1,j) L_j^{(a)} + \min(i+1,j) L_j^{(a)},
\]
and hence
\[
2p_i^{(a)} \geq p_{i-1}^{(a)} + p_{i+1}^{(a)}.\qedhere
\]
\end{proof}

We also show the following proposition for generalized types.

\begin{prop}
\label{prop:ep_phi}
Consider a rigged configuration $(\nu, J) \in \RC(\infty)$. Fix some $a \in I$ and consider the generalization of the vacancy numbers given in Lemma~{\upshape\ref{lemma:convexity}} such that $p_{\infty}^{(a)} = \inner{h_a}{\wt(\nu, J)}$. Let $x$ be the smallest label of $(\nu, J)^{(a)}$ and $s = \min(0, x)$. Then we have
\[
\varepsilon_a(\nu, J)  = -s, \ \ \ \ \ 
\varphi_a(\nu, J)  = p_{\infty}^{(a)} - s.
\]
\end{prop}

\begin{proof}
The proof that $\varphi_a(\nu, J) = p_{\infty}^{(a)} - s$ follows that given in~\cite[Lemma 3.6]{S06} and relies on the convexity statement of Lemma~\ref{lemma:convexity}. The statement for $\varepsilon_a(\nu,J)$ follows from $p_{\infty}^{(a)} = \inner{h_a}{\wt(\nu,J)} = \varphi_a(\nu,J) - \varepsilon_a(\nu,J)$ (or \cite[Thm.~3.8]{Sakamoto13}).
\end{proof}

Note that the proof of Theorem~\ref{S06-thm} given in~\cite{S06} is based on the Stembridge axioms~\cite{Stembridge03} and does not use the condition that the crystal of finite type. However it does rely upon Proposition~\ref{prop:ep_phi} for simply-laced types (this is contained in~\cite{S06}).  Hence the proof holds for arbitrary simply-laced types, and it gives a rigged configuration model for highest weight modules in arbitrary simply-laced types. Similarly, the proof of Theorem~\ref{thm:RCinf} does not use any assumption that the Kac-Moody algebra be of finite type, so our result extends to arbitrary simply-laced types.

\begin{thm}\label{RCinf_simply-laced}
Let $\g$ be of simply-laced type.  Then there exists a $U_q(\g)$-crystal isomorphism $B(\infty) \cong \RC(\infty)$ which sends $u_\infty \mapsto (\nu_\emptyset,J_\emptyset)$.
\end{thm}

\begin{ex}
Consider the hyperbolic Kac-Moody algebra $H_1^{(4)}$ (see \cite{Cetal:10} for the notation and list of Dynkin diagrams), whose Dynkin diagram is the complete graph on four vertices.
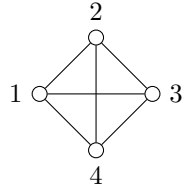
\[
\begin{tikzpicture}[scale=.75,baseline=-2]
\node[circle,draw,scale=.6,label={west:$1$}] (1) at (0,0) {};
\node[circle,draw,scale=.6,label={north:$2$}] (2) at (1,1) {};
\node[circle,draw,scale=.6,label={south:$4$}] (4) at (1,-1){};
\node[circle,draw,scale=.6,label={east:$3$}] (3) at (2,0) {};
\path[-]
 (1) edge (2)
 (1) edge (3)
 (1) edge (4)
 (2) edge (3)
 (2) edge (4)
 (3) edge (4);
\end{tikzpicture}
\]
Then the partitions are enumerated as $(\nu^{(1)},\nu^{(2)},\nu^{(3)},\nu^{(4)})$ and 
\[
f_4f_2^2f_1f_3f_4^3f_2f_1(\nu_\emptyset,J_\emptyset) = 
\begin{tikzpicture}[scale=.35,baseline=-20]
 \rpp{1,1}{2,1}{1,1}
 \begin{scope}[xshift=5cm]
 \rpp{2,1}{-1,1}{0,1}
 \end{scope}
 \begin{scope}[xshift=11cm]
 \rpp{1}{2}{4}
 \end{scope} 
 \begin{scope}[xshift=16cm]
 \rpp{3,1}{-1,-1}{-2,1}
 \end{scope}
\end{tikzpicture}.
\]
\end{ex}

\section{Extending Theorem~\ref{thm:RCinf} to non-simply-laced Lie algebras}
\label{sec:non-simply-laced}

\subsection{Virtual crystals}
\label{sec:virtual_crystals}
In this section, $\g$ denotes an affine Kac-Moody algebra with classical subalgebra $\g_0$.
\begin{table}[t]
\onehalfspacing
\[
\begin{array}{|c|c|c|c|c|}\hline
\text{type of } \g & C_n^{(1)}, A_{2n}^{(2)}, A_{2n}^{(2)\dagger}, D_{n+1}^{(2)} & B_n^{(1)}, A_{2n-1}^{(2)} & E_6^{(2)}, F_4^{(1)} & G_2^{(1)}, D_4^{(3)} \\\hline
\text{type of } \virtual{\g} & A_{2n-1}^{(1)} & D_{n+1}^{(1)} & E_6^{(1)} & D_4^{(1)} \\\hline
\end{array}
\]
\caption{Well-known embeddings $\g \lhook\joinrel\longrightarrow \virtual{\g}$ of affine Kac-Moody algebras by type as given in \cite{JM85} $(n\neq 1)$.}\label{table:affine_embeddings}
\end{table}
Fix one of the embeddings $\g\lhook\joinrel\longrightarrow \virtual{\g}$ from Table \ref{table:affine_embeddings}, so that $\virtual{\g}$ is simply-laced with index set denoted by $\virtual{I}$. Let $\Gamma$ be the Dynkin diagram of $\g$ and $\virtual{\Gamma}$ be the Dynkin diagram of $\widehat{\g}$.\footnote{From now on, if $S$ is an object associated with $\g$, then $\virtual{S}$ will denote the corresponding object associated with $\virtual{\g}$ under the appropriate embedding listed above.} These embeddings arise from the diagram foldings $\phi \colon \virtual{\Gamma} \folding \Gamma$. We also have to define additional data $\gamma = (\gamma_a)_{a \in I}$ in the following way. 
\begin{enumerate}
\item Suppose $\Gamma$ has a unique arrow.  Removing the edge with this unique arrow leaves two connected components.
\begin{enumerate}
  \item Suppose the arrow points towards the component of the special node $0$. Then $\gamma_a = 1$ for all $a \in I$.
  \item Suppose instead the arrow points away from the component of the special node $0$.  Then $\gamma_a$ is the order of $\phi$ for all $a$ in the component of $0$ after removing the arrow.  For $a$ in the component not containing $0$, set $\gamma_a = 1$.
\end{enumerate}
\item If $\Gamma$ has two arrows, then $\Gamma$ embeds into the Dynkin diagram of $A_{2n-1}^{(1)}$.  Then $\gamma_a = 1$ for all $1 \leq a \leq n-1$, and for $a \in \{0, n\}$, we have $\gamma_a = 2$ if the arrow points away from $a$ and $\gamma_a = 1$ otherwise.
\end{enumerate}

We have two special cases of the above for types $A_1^{(1)}$ and $A_2^{(2)}$. For type $A_1^{(1)}$, we consider the diagram folding of $A_3^{(1)}$ given by $\phi^{-1}(0) = \{0, 2\}$ and $\phi^{-1}(1) = \{1, 3\}$ and $\gamma_0 = \gamma_1 = 1$. For type $A_2^{(2)}$, we consider the diagram folding of $D_4^{(1)}$ given by $\phi^{-1}(0) = \{0,1,3,4\}$ and $\phi^{-1}(1) = \{2\}$ and $\gamma_0 = 1$ and $\gamma_1 = 4$.

The embeddings in Table \ref{table:affine_embeddings} yield natural embeddings $\Psi \colon P \longrightarrow \widehat{P}$ of weight lattices as
\[
\Lambda_a \mapsto \gamma_a \sum_{b \in \phi^{-1}(a)} \widehat{\Lambda}_b
\quad \quad
\text{ and }
\quad \quad
\alpha_a \mapsto \gamma_a \sum_{b \in \phi^{-1}(a)} \virtual{\alpha}_b.
\]
This implies that $\Psi(\delta) = c_0 \gamma_0 \virtual{\delta}$, where $\delta$ (resp.\ $\virtual{\delta}$) is the minimal positive imaginary root in $P$ (resp.\ $\virtual{P}$).

\begin{remark}
\label{rem:A2dual}
There is another folding of $D_4^{(1)}$ to obtain $A_2^{(2)}$ by setting $\phi^{-1}(0) = \{2\}$ and $\phi^{-1}(1) = \{0,1,3,4\}$, but with $\gamma_0 = \gamma_1 = 1$. Since $0 \notin \phi^{-1}(0)$, we have $\Psi(\delta) \neq c_0 \gamma_0 \virtual{\delta}$. This implies $\Psi(\delta) = c_{\phi(0)} \gamma_{\phi(0)} \virtual{\delta}$; i.e., we want the coefficients of $\virtual{\delta}$ to correspond to the image of $0$ under the diagram folding.  Alternatively we could consider this as a folding of $A_2^{(2)\dagger}$, which is the same as the Dynkin diagram of $A_2^{(2)}$ but with the labels of nodes interchanged (with $1$ as the affine node).
\end{remark}

Next we restrict our focus to untwisted types; that is, we only consider
\begin{equation}\label{eq:untwisted_embeddings}
\begin{aligned}
C_n^{(1)} & \lhook\joinrel\longrightarrow A_{2n-1}^{(1)},
& B_n^{(1)} & \lhook\joinrel\longrightarrow D_{n+1}^{(1)},
\\ F_4^{(1)} & \lhook\joinrel\longrightarrow E_6^{(1)},
& G_2^{(1)} & \lhook\joinrel\longrightarrow D_4^{(1)}.
\end{aligned}
\end{equation}
When restricting to the classical subalgebras from~\eqref{eq:untwisted_embeddings}, we get the embeddings
\begin{equation}\label{eq:classical_embeddings}
\begin{aligned}
C_n & \lhook\joinrel\longrightarrow A_{2n-1},
& B_n & \lhook\joinrel\longrightarrow D_{n+1},
\\ F_4 & \lhook\joinrel\longrightarrow E_6,
& G_2 & \lhook\joinrel\longrightarrow D_4,
\end{aligned}
\end{equation}
via diagram foldings.

If $\g_0 \lhook\joinrel\longrightarrow \virtual{\g}_0$ is one of the embeddings from~\eqref{eq:classical_embeddings}, then it induces an injection $v \colon B(\lambda) \lhook\joinrel\longrightarrow B(\virtual{\lambda})$ as sets, where $\Psi(\lambda) = \virtual{\lambda}$. However, there is additional structure on the image under $v$ as a \emph{virtual crystal}, where $e_a$ and $f_a$ are defined on the image as
\begin{equation}
\label{eq:virtual_crystal_ops}
e^v_a = \prod_{b \in \phi^{-1}(a)} \virtual{e}_b^{\,\gamma_a}
\quad \quad
\text{ and }
\quad \quad
f^v_a = \prod_{b \in \phi^{-1}(a)} \virtual{f}_b^{\,\gamma_a},
\end{equation}
respectively, and they commute with $v$~\cite{baker2000,OSS03III,OSS03II}. These are known as the \emph{virtual Kashiwara (crystal) operators}. It is shown in~\cite{K96} that for any $a \in I$ and $b,b^{\prime} \in \phi^{-1}(a)$ we have $e_b e_{b^{\prime}} = e_{b^{\prime}} e_b$ and $f_b f_{b^{\prime}} = f_{b^{\prime}} f_b$ as operators (recall that $b$ and $b^{\prime}$ are not connected), so both $e^v_a$ and $f^v_b$ are well-defined. The inclusion map $v$ also satisfies the following commutative diagram.
\begin{equation}
\label{eq:virtual_weight}
\begin{tikzpicture}[xscale=4, yscale=1.5, text height=1.9ex, text depth=0.25ex, baseline=.75cm]
\node (1) at (0,1) {$B(\lambda)$};
\node (2) at (1,1) {$B(\virtual{\lambda})$};
\node (3) at (0,0) {$P$};
\node (4) at (1,0) {$\virtual{P}$};
\path[right hook->,font=\scriptsize]
 (1) edge node[above,inner sep=2]{$v$} (2)
 (3) edge node[below,inner sep=1]{$\Psi$} (4);
\path[->,font=\scriptsize]
 (1) edge node[left] {$\wt$} (3)
 (2) edge node[right]{$\virtual{\wt}$} (4);
\end{tikzpicture}
\end{equation}
In~\cite{baker2000}, it was shown that this defines a $U_q(\g)$-crystal structure on the image of $v$.  More generally, we define a virtual crystal as follows.

\begin{dfn}\label{def:virtual}
Consider any symmetrizable types $\g$ and $\virtual{\g}$ with index sets $I$ and $\virtual{I}$, respectively. Let $\phi \colon \virtual{I} \longrightarrow I$ be a surjection such that $b$ is not connected to $b^{\prime}$ for all $b,b^{\prime} \in \phi^{-1}(a)$ and $a \in I$. Let $\virtual{B}$ be a $U_q(\virtual{\g})$-crystal and $V \subseteq \virtual{B}$.  Let $\gamma = (\gamma_a \in \ZZ_{>0} \mid a \in I)$. A \emph{virtual crystal} is the quadruple $(V, \virtual{B}, \phi, \gamma)$ such that $V$ has an abstract $U_q(\g)$-crystal structure defined using the Kashiwara operators $e_a^v$ and $f_a^v$ from \eqref{eq:virtual_crystal_ops} above,
\begin{align*}
\varepsilon_a &:= \gamma_a^{-1} \virtual{\varepsilon}_b, & 
\varphi_a &:= \gamma_a^{-1} \virtual{\varphi}_b, &\text{ for all } b\in \phi^{-1}(a),
\end{align*}
and $\wt := \Psi^{-1} \circ \virtual{\wt}$.
\end{dfn}

\begin{remark}
The definition of $\varepsilon_a$ and $\varphi_a$ forces all of our virtual crystals to be \emph{aligned}, as defined in \cite{OSS03III,OSS03II}.
\end{remark}

We say $B$ \emph{virtualizes} in $\virtual{B}$ if there exists a $U_q(\g)$-crystal isomorphism $v \colon B \longrightarrow V$.  The resulting isomorphism is called the \emph{virtualization map}. We denote the quadruple $(V,\virtual{B},\phi,\gamma)$ simply by $V$ when there's no risk of confusion.

The virtualization map $v$ from rigged configurations of type $\g_0$ to rigged configurations of type $\virtual{\g}_0$ is defined by
\begin{equation}
\label{eq:virtual_rc}
\virtual{m}_{\gamma_a i}^{(b)} = m_i^{(a)}, \quad \virtual{J}_{\gamma_a i}^{(b)} = \gamma_a J_i^{(a)},
\end{equation}
for all $b \in \phi^{-1}(a)$.  A $U_q(\g_0)$-crystal structure on rigged configurations is defined by using virtual crystals \cite{OSS03II}. Moreover, we use Equation~\eqref{eq:virtual_rc} to describe the virtual image of the type $\g_0$ rigged configurations into type $\virtual{\g}_0$ rigged configurations. Explicitly $(\virtual{\nu}, \virtual{J}) \in V$ if and only if
\begin{enumerate}
\item $\virtual{m}_i^{(b)} = \virtual{m}_i^{(b^{\prime})}$ and $\virtual{J}_i^{(b)} = \virtual{J}_i^{(b^{\prime})}$ for all $b, b^{\prime} \in \phi^{-1}(a)$,
\item $\virtual{m}^{(b)}_i \in \gamma_a \ZZ$ and $\virtual{J}^{(b)}_i \in \gamma_a \ZZ$ for all $b \in \phi^{-1}(a)$, and
\item $\virtual{m}^{(b)}_i = 0$ and $\virtual{J}^{(b)}_i = 0$ for all $j \notin \gamma_a \ZZ$ for all $b \in \phi^{-1}(a)$.
\end{enumerate}

\begin{ex}
Consider the rigged configuration in type $C_2$
\[
\begin{tikzpicture}[scale=.35,anchor=top]
 \rpp{2}{1}{1}
 \begin{scope}[xshift=6cm]
 \rpp{1,1}{-1,-1}{-1,-1}
 \end{scope}
\end{tikzpicture}
\]
with $L_1^{(1)} = L_1^{(2)} = 1$, all other $L_i^{(a)} = 0$, and weight $\Lambda_1 - \Lambda_2$. The corresponding virtual rigged configuration in type $A_3$ is
\[
\begin{tikzpicture}[scale=.35,anchor=top]
 \rpp{2}{1}{1}
 \begin{scope}[xshift=6cm]
 \rpp{2,2}{-2,-2}{-2,-2}
 \end{scope}
 \begin{scope}[xshift=12cm]
 \rpp{2}{1}{1}
 \end{scope}
\end{tikzpicture}
\]
with $L_1^{(1)} = L_1^{(3)} = L_2^{(2)} = 1$, all other $L_i^{(a)} = 0$, and weight $\Lambda_1 + \Lambda_3 - 2 \Lambda_2$.
\end{ex}

\begin{remark}
There exist rigged configurations for $U_q(\g_0)$-crystals when $\g$ is of twisted affine type by considering $U_q^{\prime}(\g)$ crystals; however, we omit those here in order to avoid confusion as we will be considering rigged configurations for $U_q(\g)$-crystals in the sequel. In particular, for type $A_{2n}^{(2)\dagger}$, the riggings of $\nu^{(n)}$ are in $\frac{1}{2}\ZZ$. See~\cite{OSS03II} for more information.
\end{remark}

We note that it is sufficient to consider single tensor factors by the following proposition.

\begin{prop}[{\cite[Prop.~6.4]{OSS03III}}]
\label{prop:aligned_tensor}
Virtual crystals form a tensor category.
\end{prop}

Although \cite{OSS03III} is concerned with $U_q'(\g)$-crystals, the proof of Proposition \ref{prop:aligned_tensor} does not use the $U_q^{\prime}(\g)$-crystals condition, but instead is a statement about the tensor product rule. It has been cited as above in other papers; e.g., Proposition~3.3 of~\cite{OSS03II}.

\subsection{Extending Theorem~\ref{thm:RCinf} to all finite types}
In this section we assume $\g$ is of non-simply-laced finite type. For the vacancy numbers, we just consider this as the classical subcrystal in the corresponding untwisted affine type.  Again, let $\RC(\infty)$ be the set generated by $(\nu_\emptyset,J_\emptyset)$ and $e_a,f_a$ for $a\in I$,  where $e_a$ and $f_a$ are defined as in Section \ref{sec:simplyfinite}.

Proposition~\ref{prop:rc_nsl_virtual} and Theorem~\ref{thm:rc_nsl} below are proven in~\cite{SS14} for all finite types.  We will require these results in the sequel.

\begin{prop}
\label{prop:rc_nsl_virtual}
The crystal $\RC(\lambda)$ virtualizes in $\RC(\virtual{\lambda})$.
\end{prop}

\begin{thm}
\label{thm:rc_nsl}
Let $\g$ be of finite type. We have $\RC(\lambda) \iso B(\lambda)$.
\end{thm}

\begin{remark}
Note the proof of Theorem \ref{thm:rc_nsl} uses the fact that $B(\lambda)$ virtualizes in $B(\virtual{\lambda})$ in finite types~\cite{baker2000, OSS03III, OSS03II}.
\end{remark}

By combining the virtualization results above with the method of proof given for Theorem~\ref{thm:RCinf}, we may extend Theorem~\ref{thm:RCinf} to include non-simply-laced finite types.

\begin{thm}
\label{thm:RCinf_nsl}
Let $\g$ be of any finite type.  Then there exists a $U_q(\g)$-crystal isomorphism $\RC(\infty) \iso B(\infty)$ such that $(\nu_{\emptyset}, J_\emptyset) \mapsto u_{\infty}$.
\end{thm}

\begin{proof}
The proof of Theorem~\ref{thm:RCinf} holds here by following Section~\ref{sec:simplyfinite} and using Theorem~\ref{thm:rc_nsl} in place of Theorem~\ref{S06-thm}.
\end{proof}

We also have the virtualization of $B(\infty)$ crystals.

\begin{prop}
\label{prop:virtual_Binf}
Let $\g$ be of any finite type. The $U_q(\g)$-crystal $B(\infty)$ virtualizes in the $U_q(\virtual{\g})$-crystal $\virtual{B}(\infty)$.
\end{prop}

\begin{proof}
This follows immediately from the fact that the diagram
\[
\begin{tikzpicture}[xscale=4, yscale=1.5, text height=1.9ex, text depth=0.25ex, baseline=.75cm]
\node (1) at (0,1.5) {$T_{-\lambda} \otimes B(\lambda)$};
\node (2) at (1,1.5) {$T_{-\lambda-\mu}\otimes B(\lambda+\mu)$};
\node (3) at (0,0) {$T_{-\virtual\lambda}\otimes B(\virtual\lambda)$};
\node (4) at (1,0) {$T_{-\virtual{\lambda}-\virtual\mu} \otimes B(\virtual\lambda + \virtual\mu)$};
\path[->,font=\scriptsize]
 (1) edge node[above,outer sep=2]{$I_{\lambda+\mu,\mu}$} (2)
 (3) edge node[below,inner sep=2]{$I_{\virtual\lambda+\virtual\mu,\virtual\mu}$} (4);
\path[right hook->,font=\scriptsize]
 (1) edge node[left] {} (3)
 (2) edge node[right]{} (4);
\end{tikzpicture}
\]
commutes.
\end{proof}

\subsection{Recognition Theorem}\label{sec:recognition_theorem}

From the above, we see that we only need to know the factors $(\gamma_a)_{a \in I}$ in order to show that we get a virtualization of the $U_q(\g)$-crystal of rigged configurations into a $U_q(\virtual{\g})$-crystal by Equation~\eqref{eq:virtual_rc}. Thus we make the following conjecture.

\begin{conj}
\label{conj:foldings}
Let $\g$ be obtained via a diagram folding $\phi$ of a simply-laced type $\virtual{\g}$. There exists $(\gamma_a)_{a \in I}$ such that $\RC(\lambda)$ virtualizes in $\RC(\virtual{\lambda})$ by Equation~{\upshape\ref{eq:virtual_rc}}.
\end{conj}

We have this for all finite and affine types using the foldings given in Table~\ref{table:affine_embeddings}. We can also show this for all rank $2$ with Cartan matrix
\[
\begin{pmatrix}
2 & x \\
y & 2
\end{pmatrix}
\]
by considering a diagram folding of $K_{x,y}$, the complete bipartite graph on $x$ and $y$ nodes, with $\gamma_1 = \gamma_2 = 1$. In such foldings, it is easy to see that Conjecture~\ref{conj:foldings} holds from Equation~\eqref{eq:vacancy_numbers}. In fact, we believe there exists a $\virtual{\g}$ such that $\gamma_a = 1$ for all $a \in I$, and we call such a folding natural.

In their development of the geometric construction of the crystal basis, Kashiwara and Saito~\cite{KS97} established a recognition theorem for the crystal $B(\infty)$ valid for all symmetrizable Kac-Moody types.  In this section, we will recall the recognition theorem with appropriate definitions and extend Theorem~\ref{thm:RCinf} to all Kac-Moody algebras satisfying Conjecture~\ref{conj:foldings} using the recognition theorem.

\begin{remark}
A straightforward check shows that Proposition~\ref{prop:rc_nsl_virtual} holds in our affine setting, which requires Lemma~\ref{lemma:convexity}, and so Conjecture~\ref{conj:foldings} is true in affine types.
\end{remark}

\begin{remark}
A priori, we do not have that $\RC(\lambda) \iso B(\lambda)$ for arbitrary symmetrizable types, as there is no equivalent version of Table \ref{table:affine_embeddings} which would give the analogous statement to Theorem~\ref{thm:rc_nsl}. Therefore we must change our techniques to show that the crystal $\RC(\infty) \iso B(\infty)$ by using the $B(\infty)$ recognition theorem given in~\cite{KS97}. Nevertheless,  we will be able to show that those $\RC(\lambda)$ carved out of $\RC(\infty)$ are isomorphic to $B(\lambda)$ in Section \ref{sec:projection}.
\end{remark}

From this viewpoint, it would be natural to restrict our attention for affine foldings from Table \ref{table:affine_embeddings} given by
\begin{equation}\label{eq:twisted_embeddings}
\begin{aligned}
D_{n+1}^{(2)} & \lhook\joinrel\longrightarrow A_{2n-1}^{(1)},
& A_{2n-1}^{(2)} & \lhook\joinrel\longrightarrow D_{n+1}^{(1)},
\\ E_6^{(2)} & \lhook\joinrel\longrightarrow E_6^{(1)},
& D_4^{(3)} & \lhook\joinrel\longrightarrow D_4^{(1)},
\end{aligned}
\end{equation}
as these foldings satisfy $\gamma_a = 1$ for all $a\in I$. The corresponding classical foldings from~\eqref{eq:classical_embeddings} are given by
\begin{equation}
\begin{aligned}
B_n & \lhook\joinrel\longrightarrow A_{2n-1},
& C_n & \lhook\joinrel\longrightarrow D_{n+1},
\\ F_4 & \lhook\joinrel\longrightarrow E_6,
& G_2 & \lhook\joinrel\longrightarrow D_4.
\end{aligned}
\end{equation}
We should also note that we can get natural foldings of the other (non-degenerate) affine types by
\begin{align*}
B_n^{(1)}, A_{2n}^{(2)} & \lhook\joinrel\longrightarrow D_{2n+1}^{(1)}
& C_n^{(1)} & \lhook\joinrel\longrightarrow D_{n+1}^{(1)}
\end{align*}
As in Remark~\ref{rem:A2dual}, we have $\Psi(\delta) = c_{\phi(0)} \gamma_{\phi(0)} \virtual{\delta}$.

\begin{dfn}
\label{def:elementary_crystal}
Let $\g$ be a symmetrizable Kac-Moody algebra and fix $a \in I$. Define $\ZZ_{(a)} = \{ z_a(m) \mid m \in \ZZ \}$ with the abstract $U_q(\g)$-crystal structure given by
\begin{gather*}
\wt\bigl(z_a(m)\bigr) = m \alpha_a, \quad 
\varphi_a\bigl(z_a(m)\bigr) = m, 
\quad \varepsilon_a\bigl(z_a(m)\bigr) = -m, \\
\varphi_b\bigl(z_a(m)\bigr) = \varepsilon_b\bigl(z_a(m)\bigr) = -\infty \text{ for } a \neq b,\\
e_a z_a(m) = z_a(m+1), \quad
f_a z_a(m) = z_a(m-1), \\
e_b z_a(m) = f_b z_a(m) = 0 \text{ for } a \neq b.
\end{gather*}
The crystal $\ZZ_{(a)}$ is called an \emph{elementary crystal}.
\end{dfn}

\begin{remark}
The crystal $\ZZ_{(a)}$ was originally denoted by $B_i$ in \cite{K93}.
\end{remark}

We must first prove a technical lemma about the virtual elementary crystals.

\begin{lemma}
\label{lemma:virtual_elementary}
Let $\g$ be a Kac-Moody algebra satisfying Conjecture~{\upshape\ref{conj:foldings}}. Let $\phi$ be the diagram folding with scaling factors $(\gamma_a)_{a \in I}$. Fix some $a \in I$. The elementary crystal $\ZZ_{(a)}$ virtualizes in $\virtual{\ZZ}_{(a)} = \bigotimes_{b \in \phi^{-1}(a)} \ZZ_{(b)}$ (for any order of the factors) with the virtualization map $v_{(a)}$ defined by
\[
z_a(m) \mapsto \bigotimes_{b \in \phi^{-1}(a)} z_b(\gamma_a m).
\]
\end{lemma}

\begin{proof}
If $\virtual{\ZZ}_{(a)} = \ZZ_{(b)}$ where $\{b\} = \phi^{-1}(a)$, then it is easy to see the claim is true from Definition~\ref{def:elementary_crystal}.

Now we assume $\virtual{\ZZ}_{(a)} = \ZZ_{(b_2)} \otimes \ZZ_{(b_1)}$ where $\{b_1, b_2\} = \phi^{-1}(a)$ and $b_1\neq b_2$. If $b \notin \phi^{-1}(a)$, then
\[
\varepsilon_b\bigl(z_{b_2}(\gamma_a m) \otimes z_{b_1}(\gamma_a m)\bigr) = \max(-\infty, -\infty - \inner{h_b}{\gamma_a m \alpha_{b_2}}) = -\infty.
\]
If $b = b_2$, then we have
\begin{align*}
\varepsilon_b\bigl(z_{b_2}(\gamma_a m) \otimes z_{b_1}(\gamma_a m)\bigr) 
&= \max(-\gamma_a m, -\infty - \inner{h_b}{\gamma_a m \alpha_{b_2}}) \\
&= -\gamma_a m \\
&= \gamma_a \varepsilon_a\bigl( z_a(m) \bigr)
\end{align*}
since $-\infty + k = -\infty$ for any finite number $k$. If $b = b_1$, then we have
\begin{align*}
\varepsilon_b\bigl(z_{b_2}(\gamma_a m) \otimes z_{b_1}(\gamma_a m)\bigr) 
&= \max(-\infty, -\gamma_a m - \inner{h_b}{\gamma_a m \alpha_{b_2}}) \\
&= -\gamma_a m \\
&= \gamma_a \varepsilon_a\bigl( z_a(m) \bigr)
\end{align*}
since $b_1\neq b_2$. Similar statements hold for $\varphi_b\bigl(z_{b_2}(\gamma_a m) \otimes z_{b_1}(\gamma_a m)\bigr)$. From the tensor product rule, 
\begin{align*}
e_b\bigl(z_{b_2}(\gamma_a m) \otimes z_{b_1}(\gamma_a m)\bigr)
&= 
\begin{cases}
z_{b_2}(\gamma_a m) \otimes e_b\bigl(z_{b_1}(\gamma_a m)\bigr) & \text{ if } b = b_1, \\
e_b\bigl(z_{b_2}(\gamma_a m)\bigr) \otimes z_{b_1}(\gamma_a m) & \text{ if } b = b_2, \\
0 & \text{ otherwise,}
\end{cases}\\
&= 
\begin{cases}
z_{b_2}(\gamma_a m) \otimes z_{b_1}\bigl(\gamma_a(m+1)\bigr) & \text{ if } b = b_1, \\
z_{b_2}\bigl(\gamma_a (m+1)\bigr) \otimes z_{b_1}(\gamma_a m) & \text{ if } b = b_2, \\
0 & \text{ otherwise,}
\end{cases}
\end{align*}
and
\begin{align*}
f_b\bigl(z_{b_2}(\gamma_a m) \otimes z_{b_1}(\gamma_a m)\bigr)
&= 
\begin{cases}
z_{b_2}(\gamma_a m) \otimes f_b\bigl(z_{b_1}(\gamma_a m)\bigr) & \text{ if } b = b_1, \\
f_b\bigl(z_{b_2}(\gamma_a m)\bigr) \otimes z_{b_1}(\gamma_a m) & \text{ if } b = b_2, \\
0 & \text{ otherwise,}
\end{cases}\\
&= 
\begin{cases}
z_{b_2}(\gamma_a m) \otimes z_{b_1}\bigl(\gamma_a(m-1)\bigr) & \text{ if } b = b_1, \\
z_{b_2}\bigl(\gamma_a (m-1)\bigr) \otimes z_{b_1}(\gamma_a m) & \text{ if } b = b_2, \\
0 & \text{ otherwise.}
\end{cases}
\end{align*}
Thus we have
\begin{align*}
(e^v_a \circ v)\bigl( z_a(m) \bigr) & = e_{b_1}^{\gamma_a} e_{b_2}^{\gamma_a}\bigl(z_{b_2}(\gamma_a m) \otimes z_{b_1}(\gamma_a m)\bigr)
\\ & = z_{b_2}\bigl( \gamma_a (m+1) \bigr) \otimes z_{b_1}\bigl( \gamma_a (m+1) \bigr)
\\ & = v\bigl( z_a(m+1) \bigr) \\
&= v\bigl(e_a z_a(m)\bigr),
\end{align*}
and $(\widehat{e}_{a^{\prime}} \circ v)\bigl( z_a(m) \bigr) = 0 = v\bigl( e_{a^{\prime}} z_a(m) \bigr)$ for $a^{\prime} \neq a$. Similar statements can be shown for $f_a$ and $f_{a^{\prime}}$ for $a^{\prime} \neq a$. Lastly
\[
\wt\bigl(z_{b_2}(\gamma_a m) \otimes z_{b_1}(\gamma_a m)\bigr) = \gamma_a m (\alpha_{b_2} + \alpha_{b_1}) = \virtual{\wt}(z_a(m)).
\]
Therefore $\ZZ_{(a)}$ virtualizes in $\ZZ_{(b_2)} \otimes \ZZ_{(b_1)}$ with virtualization map $v$.  It is clear that it is independent of the ordering. Moreover, we may generalize to the case of finitely many tensor factors using induction and associativity of the tensor product with a similar argument as above.
\end{proof}

\begin{thm}[Recognition Theorem {\cite[Prop.~3.2.3]{KS97}}]
\label{thm:binf_recog}
Let $\g$ be a symmetrizable Kac-Moody algebra, $B$ be an abstract $U_q(\g)$-crystal, and $x_0$ be an element of $B$ with weight zero.  Assume the following conditions.
\begin{enumerate}
\item\label{rec:1} $\wt(B) \subset Q_-$.
\item\label{rec:2} $x_0$ is the unique element of $B$ with weight zero.
\item\label{rec:3} $\varepsilon_a(x_0) = 0$ for all $a \in I$.
\item\label{rec:4} $\varepsilon_a(x) \in \ZZ$ for all $x \in B$ and $a \in I$.
\item\label{rec:5} For every $a \in I$, there exists a strict crystal embedding $\Psi_a \colon B \longrightarrow \ZZ_{(a)} \otimes B$.
\item\label{rec:6} $\Psi_a(B) \subset \{ f_a^m z_a(0) \mid m \geq 0 \} \times B$.
\item\label{rec:7} For any $x \in B$ such that $x \neq x_0$, there exists $a \in I$ such that $\Psi_a(x) = f_a^m z_a(0) \otimes x'$ with $m > 0$ and $x' \in B$.
\end{enumerate}
Then $B$ is isomorphic to $B(\infty)$.
\end{thm}

\begin{lemma}
\label{lemma:top+f}
Assume $\g$ satisfies Conjecture~{\upshape\ref{conj:foldings}}.  Then the crystal $\RC(\lambda)$ is generated by $(\nu_\emptyset,J_\emptyset)$ and $f_a$ for all $a \in I$.
\end{lemma}

\begin{proof}
By assumption, $\RC(\lambda)$ virtualizes in $\RC(\virtual{\lambda})$.  Since $\RC(\virtual{\lambda}) \cong B(\virtual{\lambda})$ and $B(\virtual{\lambda})$ is generated by its highest weight vector and $\virtual{f}_a$ for all $a \in \virtual{I}$, the statement follows.
\end{proof}

\begin{thm}
\label{thm:RCinf_folding}
Let $\g$ be a Kac-Moody algebra satisfying Conjecture~{\upshape\ref{conj:foldings}}. Then $\RC(\infty) \iso B(\infty)$ as $U_q(\g)$-crystals.
\end{thm}

\begin{proof}
Let $\virtual{\RC}(\infty)$ denote the rigged configuration realization of the crystal $\virtual{B}(\infty)$ corresponding to the simply-laced Kac-Moody algebra $\virtual{\g}$ coming from Theorem~\ref{RCinf_simply-laced}, so that 
\[
\virtual{\RC}(\infty) = \varinjlim_{\lambda\in P^+} \bigl(T_{-\lambda} \otimes \RC(\virtual{\lambda})\bigr).
\]
From Conjecture~\ref{conj:foldings}, we have
\[
\RC(\infty) = \varinjlim_{\lambda\in P^+} \bigl(T_{-\lambda} \otimes \RC(\lambda)\bigr),
\]
for reasons similar to the justification of Theorem~\ref{thm:RCinf_nsl}.  Hence $\RC(\infty)$ virtualizes in $\virtual{\RC}(\infty)$ as in Proposition~\ref{prop:virtual_Binf}.  It remains to show that $\RC(\infty) \cong B(\infty)$ as $U_q(\g)$-crystals.

We note that~(\ref{rec:1}) and~(\ref{rec:2}) are satisfied from Equation~\eqref{Binf_wt} where $x_0 = (\nu_{\emptyset}, J_{\emptyset})$. Condition (\ref{rec:3}) is satisfied directly by the definition of $(\nu_\emptyset,J_\emptyset)$, while (\ref{rec:4}) follows from the definition of $\varepsilon_a$ on $\RC(\infty)$.  The remaining properties require virtualization.

Let $v_{(a)}$ denote the virtualization map from Lemma~\ref{lemma:virtual_elementary}. Now for each $a \in I$, define a crystal morphism $\Psi_a \colon \RC(\infty) \longrightarrow \ZZ_{(a)} \otimes \RC(\infty)$ in the following way.
Consider the following commutative diagram.
\[
\begin{tikzpicture}[xscale=5, yscale=2, text height=1.8ex, text depth=0.25ex]
\node (1) at (0,1) {$\RC(\infty)$};
\node (2) at (1,1) {$\virtual{\RC}(\infty)$};
\node (3) at (0,0) {$\ZZ_{(a)} \otimes \RC(\infty)$};
\node (4) at (1,0) {$\virtual{\ZZ}_{(a)} \otimes \virtual{\RC}(\infty)$};
\path[->,font=\scriptsize]
 (1) edge node[above,inner sep=2]{$v$} (2)
 (1) edge node[left] {$\Psi_a$} (3)
 (2) edge node[right]{$\widehat{\Psi}_a$} (4)
 (3) edge node[below,inner sep=1]{$v_{(a)} \otimes v$} (4);
\end{tikzpicture}
\]
Since both rows are virtualization maps by Proposition~\ref{prop:aligned_tensor} and map on the right side is a strict embedding because $\virtual{\RC}(\infty) \cong \virtual{B}(\infty)$ by Theorem~\ref{RCinf_simply-laced}, we get a well-defined strict embedding $\Psi_a = (v_{(a)}\otimes v)^{-1}\circ\virtual{\Psi}_a \circ v$ for every $a\in I$.

For~(\ref{rec:6}), notice the crystal $\RC(\infty)$ is generated from $(\nu_{\emptyset}, J_{\emptyset})$ and $f_a$, for $a \in I$, from the direct limit characterization of $\RC(\infty)$ and Lemma \ref{lemma:top+f}. That is to say, we can write an arbitrary element $(\nu,J)$ of $\RC(\infty)$ as $(\nu,J) = f_{a_k} \cdots f_{a_1} (\nu_{\emptyset}, J_\emptyset)$ where $a_j \in I$.  Since $\Psi_a$ is strict and $f_a^k$ is a nonzero operator on both $\ZZ_{(a)}$ and $\RC(\infty)$ for all $a\in I$ and $k\ge 0$, we have $\Psi_a\bigl(\RC(\infty)\bigr) \subset \{ f_a^m z_a(0) \mid m \geq 0 \} \times \RC(\infty)$.

Finally, set $(\nu,J) = f_{a_k} \cdots f_{a_1} (\nu_{\emptyset},J_{\emptyset})$ to be an arbitrary element of $\RC(\infty)$ and take $a = a_1$. Note that $\varphi_a(\nu_\emptyset, J_\emptyset)=0$ by Equation \eqref{Binf_phi}.  Then by the tensor product rule for crystals, we have $\Psi_a\bigl(f_a(\nu_{\emptyset}, J_{\emptyset}) \bigr) = f_a z_a(0) \otimes (\nu_{\emptyset}, J_{\emptyset})$ because $\Psi_a(\nu_\emptyset,J_\emptyset) = z_a(0) \otimes (\nu_\emptyset,J_\emptyset)$. Therefore there exists some subsequence $(a_{j_1},\dots,a_{j_{k-m}})$ of $(a_1, \dotsc, a_k)$ such that $a_1 = a_{t}$, for all $t\neq j_1,\dots,j_{k-m}$, and $\Psi_a(\nu,J) = f_a^m z_a(0) \otimes f_{a_{j_t}}\cdots f_{a_{j_1}}(\nu_{\emptyset}, J_{\emptyset})$ with $m > 0$. This shows condition (\ref{rec:7}), and we have $\RC(\infty) \iso B(\infty)$ by Theorem~\ref{thm:binf_recog}.
\end{proof}

\begin{prob}
It would be interesting to find a proof which does not appeal to virtualization in order to prove (\ref{rec:5}), (\ref{rec:6}), and (\ref{rec:7}); in particular, to show that $\RC(\infty)$ is generated only by $(\nu_\emptyset,J_\emptyset)$ and $f_a$, for all $a\in I$, without appealing to virtualization.
\end{prob}

\section{Projecting from $\RC(\infty)$ to $\RC(\lambda)$}\label{sec:projection}

The goal of this section is to show that taking valid rigged configurations is equivalent to projecting to highest weight $U_q(\g)$-crystals, where $\g$ is any symmetrizable Kac-Moody type satisfying Conjecture~\ref{conj:foldings}.  Recall the one-element crystal $T_{\lambda} = \{ t_{\lambda} \}$ given in Definition~\ref{def:T_crystal}. Let $C = \{c\}$ be the one-element crystal with crystal operations defined by
\begin{displaymath}
\wt(c) = 0, \quad \varphi_a(c) = \varepsilon_a(c) = 0, \quad f_a(c) = e_a(c) = 0, \qquad a\in I.
\end{displaymath}
It is known that the connected component in $C \otimes T_{\lambda} \otimes B(\infty)$ generated by $c \otimes t_{\lambda} \otimes u_{\infty}$ is isomorphic to $B(\lambda)$.  In the setting of rigged configurations, recall that to pass from $\RC(\infty)$ to $\RC(\lambda)$, we raise the weight by $\lambda$ (equivalently we shift the vacancy numbers), which corresponds to tensoring with $T_{\lambda}$.  Next we take only valid rigged configurations, and we will show that this restriction corresponds to tensoring with the crystal $C$.

Let $\RC_{\lambda}(\infty) = T_{\lambda} \otimes \RC(\infty)$ denote the crystal associated with the Verma module with highest weight $\lambda$.   Strictly speaking, 
\[
\RC_\lambda(\infty) = \bigl\{f_{a_k}\cdots f_{a_1}\big(t_\lambda\otimes(\nu_\emptyset,J_\emptyset)\big) \mid  a_1,\dots,a_k \in I,\ k \ge 0 \bigr\},
\] 
but by an abuse of notation, we will consider $\RC_\lambda(\infty)$ as the set of all rigged configurations generated by $f_a$ $(a \in I)$ from $(\nu_{\emptyset},J_\emptyset)$ where the vacancy numbers and the weights are shifted by $\lambda$.  That is, if $\lambda = \sum_{(a,i) \in \HH} iL_i^{(a)} \Lambda_a$ is a dominant integral weight of type $\g$, then for all $i \in \ZZ_{\geq 0}$ we have 
\begin{displaymath}
p_i^{(a)}(\nu_\lambda) = \sum_{j\ge0} \min(i,j)L_j^{(a)} + p_i^{(a)}(\nu),
\qquad
\wt(\nu_\lambda,J_\lambda) = \wt(\nu,J) + \lambda,
\end{displaymath}
where $(\nu_{\lambda}, J_\lambda) \in \RC_{\lambda}(\infty)$ corresponds to $(\nu,J) \in \RC(\infty)$.

\begin{thm}
\label{thm:projection}
Let $\mathcal{C}_\emptyset$ denote the connected component of $C\otimes \RC_\lambda(\infty)$ generated by $c \otimes (\nu_\emptyset, J_\emptyset)$. 
The map $\Psi\colon \mathcal{C}_\emptyset \longrightarrow \RC(\lambda)$ sending
\begin{math}
c \otimes (\nu_\lambda, J_\lambda) \mapsto (\nu_\lambda, J_\lambda)
\end{math} 
is a $U_q(\g)$-crystal isomorphism.
\end{thm}

\begin{proof} 
Let $(\nu_{\lambda}, J_{\lambda}) \in \RC_\lambda(\infty)$ and $a \in I$.  First,
\begin{displaymath}
\wt\bigl( c \otimes (\nu_{\lambda}, J_{\lambda}) \bigr)
= \wt(c) + \wt(\nu_{\lambda}, J_{\lambda})
= \wt(\nu_{\lambda}, J_{\lambda}),
\end{displaymath}
so $\Psi$ preserves weights.  Then, 
\begin{displaymath}
\varepsilon_a\bigl( c\otimes (\nu_{\lambda}, J_{\lambda}) \bigr) = \max\bigl\{ 0, \varepsilon_a(\nu_{\lambda}, J_{\lambda}) \bigr\} = \varepsilon_a(\nu_{\lambda}, J_{\lambda}),
\end{displaymath}
since $\varepsilon_a(\nu_{\lambda}, J_{\lambda}) \geq 0$, which implies that $\Psi$ preserves $\varepsilon_a$.  
From the $\varepsilon_a\bigl(c \otimes (\nu_{\lambda}, J_{\lambda})\bigr)$ computation above, we have
\begin{align*}
\varphi_a\bigl( c\otimes (\nu_\lambda,J_\lambda)\bigr) 
&= \max\bigl\{ \varphi_a(\nu_\lambda,J_\lambda), \langle h_a, \wt(\nu_\lambda,J_\lambda)\rangle \bigr\} \\
&= \max\bigl\{ \varepsilon_a(\nu_\lambda,J_\lambda) + \langle h_a,\wt(\nu_\lambda,J_\lambda)\rangle , \langle h_a, \wt(\nu_\lambda,J_\lambda) \rangle \bigr\} \\
&= \varepsilon_a(\nu_\lambda,J_\lambda) + \langle h_a,\wt(\nu_\lambda,J_\lambda)\rangle \\
&= \varphi_a(\nu_\lambda,J_\lambda).
\end{align*}
We have $\varphi_a(\nu_\lambda,J_\lambda) = 0$ if and only if $f_a(\nu_{\lambda}, J_{\lambda}) = 0$ in $\RC(\lambda)$ because $\RC(\lambda)$ is a (lower) regular crystal.  Also if $\varphi_a(\nu_{\lambda}, J_{\lambda}) = 0$, we have
\[
f_a\bigl(c \otimes (\nu_{\lambda},J_{\lambda})\bigr) = (f_a c) \otimes (\nu_{\lambda}, J_{\lambda})= 0
\]
by the tensor product rule.  Similarly if $\varphi_a(\nu_{\lambda}, J_{\lambda}) > 0$, then
\[
f_a\bigl(c \otimes (\nu_{\lambda},J_{\lambda})\bigr) = c \otimes f_a(\nu_{\lambda}, J_{\lambda}).
\]
So $\Psi \circ f_a = f_a \circ \Psi$. Recall that $\varphi_a\bigl( c \otimes (\nu_\lambda,J_\lambda)\bigr) = \varphi_a(\nu_{\lambda}, J_{\lambda}) \geq 0$; so it follows, by the tensor product rule, that
\begin{displaymath}
\Psi\bigl(e_a\bigl( c\otimes (\nu_\lambda,J_\lambda) \bigr)\bigr) = \Psi\bigl(c \otimes e_a(\nu_\lambda,J_\lambda)\bigr) = e_a(\nu_\lambda,J_\lambda) = e_a\Psi\bigl( c \otimes (\nu_\lambda,J_\lambda)\bigr).
\end{displaymath} 
This completes the proof that $\Psi$ is a crystal isomorphism.
\end{proof}

Thus, the projection map above corresponds to eliminating those rigged configurations which are not valid; that is, $\Psi(c\otimes(\nu,J)) = 0$ if $(\nu,J)$ is not valid. Therefore Theorem~\ref{thm:RCinf_folding} implies the following.

\begin{cor}
Suppose Conjecture~{\upshape\ref{conj:foldings}} holds, then we have $\RC(\lambda) \iso B(\lambda)$.
\end{cor}

\begin{cor}
Suppose Conjecture~{\upshape\ref{conj:foldings}} holds, then the $U_q(\g)$-crystal $B(\lambda)$ virtualizes in the $U_q(\virtual{\g})$-crystal $B(\virtual{\lambda})$.
\end{cor}

We also note that Proposition~\ref{prop:ep_phi} extends to both $\RC(\infty)$ and $\RC(\lambda)$.

\begin{ex}
Consider $\RC(\Lambda_0)$ with $\g = A_2^{(1)}$.  The top of the crystal graph is shown in Figure \ref{fig:graph}.
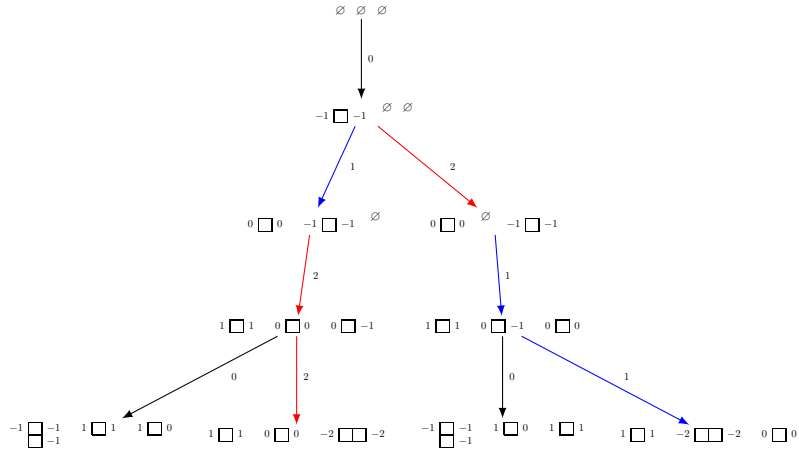
\begin{figure}[ht]
\[
\begin{tikzpicture}[>=latex,line join=bevel, xscale=.4, yscale=.5, every node/.style={scale=0.4}]
\node (11-2-200) at (667bp,16bp) [draw,draw=none] {${\begin{array}[t]{r|c|l}\cline{2-2} 1 &\phantom{|}& 1 \\ \cline{2-2} \end{array}} \quad {\begin{array}[t]{r|c|c|l}\cline{2-3} -2 &\phantom{|}&\phantom{|}& -2 \\ \cline{2-3} \end{array}} \quad {\begin{array}[t]{r|c|l}\cline{2-2} 0 &\phantom{|}& 0 \\ \cline{2-2} \end{array}}$};
  \node (///) at (340bp,336bp) [draw,draw=none] {${\emptyset}\quad{\emptyset}\quad{\emptyset}$};
  \node (11000-1) at (279bp,98bp) [draw,draw=none] {${\begin{array}[t]{r|c|l}\cline{2-2} 1 &\phantom{|}& 1 \\ \cline{2-2} \end{array}} \quad {\begin{array}[t]{r|c|l}\cline{2-2} 0 &\phantom{|}& 0 \\ \cline{2-2} \end{array}} \quad {\begin{array}[t]{r|c|l}\cline{2-2} 0 &\phantom{|}& -1 \\ \cline{2-2} \end{array}}$};
  \node (00/-1-1) at (465bp,177bp) [draw,draw=none] {${\begin{array}[t]{r|c|l}\cline{2-2} 0 &\phantom{|}& 0 \\ \cline{2-2} \end{array}} \quad{\emptyset} \quad {\begin{array}[t]{r|c|l}\cline{2-2} -1 &\phantom{|}& -1 \\ \cline{2-2} \end{array}}$};
  \node (-1-11011-1-1) at (473bp,16bp) [draw,draw=none] {${\begin{array}[t]{r|c|l}\cline{2-2} -1 &\phantom{|}& -1 \\ \cline{2-2}  &\phantom{|}& -1 \\ \cline{2-2} \end{array}} \quad {\begin{array}[t]{r|c|l}\cline{2-2} 1 &\phantom{|}& 0 \\ \cline{2-2} \end{array}} \quad {\begin{array}[t]{r|c|l}\cline{2-2} 1 &\phantom{|}& 1 \\ \cline{2-2} \end{array}}$};
  \node (1100-2-2) at (279bp,16bp) [draw,draw=none] {${\begin{array}[t]{r|c|l}\cline{2-2} 1 &\phantom{|}& 1 \\ \cline{2-2} \end{array}} \quad {\begin{array}[t]{r|c|l}\cline{2-2} 0 &\phantom{|}& 0 \\ \cline{2-2} \end{array}} \quad {\begin{array}[t]{r|c|c|l}\cline{2-3} -2 &\phantom{|}&\phantom{|}& -2 \\ \cline{2-3} \end{array}}$};
  \node (-1-1//) at (340bp,259bp) [draw,draw=none] {${\begin{array}[t]{r|c|l}\cline{2-2} -1 &\phantom{|}& -1 \\ \cline{2-2} \end{array}} \quad{\emptyset} \quad{\emptyset}$};
  \node (-1-11110-1-1) at (85bp,16bp) [draw,draw=none] {${\begin{array}[t]{r|c|l}\cline{2-2} -1 &\phantom{|}& -1 \\ \cline{2-2}  &\phantom{|}& -1 \\ \cline{2-2} \end{array}} \quad {\begin{array}[t]{r|c|l}\cline{2-2} 1 &\phantom{|}& 1 \\ \cline{2-2} \end{array}} \quad {\begin{array}[t]{r|c|l}\cline{2-2} 1 &\phantom{|}& 0 \\ \cline{2-2} \end{array}}$};
  \node (110-100) at (473bp,98bp) [draw,draw=none] {${\begin{array}[t]{r|c|l}\cline{2-2} 1 &\phantom{|}& 1 \\ \cline{2-2} \end{array}} \quad {\begin{array}[t]{r|c|l}\cline{2-2} 0 &\phantom{|}& -1 \\ \cline{2-2} \end{array}} \quad {\begin{array}[t]{r|c|l}\cline{2-2} 0 &\phantom{|}& 0 \\ \cline{2-2} \end{array}}$};
  \node (00-1-1/) at (293bp,177bp) [draw,draw=none] {${\begin{array}[t]{r|c|l}\cline{2-2} 0 &\phantom{|}& 0 \\ \cline{2-2} \end{array}} \quad {\begin{array}[t]{r|c|l}\cline{2-2} -1 &\phantom{|}& -1 \\ \cline{2-2} \end{array}} \quad{\emptyset}$};
  \draw [blue,->] (00/-1-1) ..controls (467.58bp,151.51bp) and (469.5bp,132.57bp)  .. (110-100);
  \definecolor{strokecol}{rgb}{0.0,0.0,0.0};
  \pgfsetstrokecolor{strokecol}
  \draw (478bp,136bp) node {$1$};
  \draw [red,->] (00-1-1/) ..controls (288.46bp,151.38bp) and (285.06bp,132.19bp)  .. (11000-1);
  \draw (297bp,136bp) node {$2$};
  \draw [black,->] (110-100) ..controls (473bp,76.897bp) and (473bp,57.615bp)  .. (-1-11011-1-1);
  \draw (482bp,60bp) node {$0$};
  \draw [blue,->] (-1-1//) ..controls (325.25bp,233.27bp) and (314.28bp,214.13bp)  .. (00-1-1/);
  \draw (332bp,218bp) node {$1$};
  \draw [red,->] (11000-1) ..controls (279bp,75.449bp) and (279bp,52.579bp)  .. (1100-2-2);
  \draw (288bp,60bp) node {$2$};
  \draw [black,->] (///) ..controls (340bp,317.7bp) and (340bp,298.02bp)  .. (-1-1//);
  \draw (349bp,300bp) node {$0$};
  \draw [red,->] (-1-1//) ..controls (380.77bp,232.26bp) and (413.19bp,210.99bp)  .. (00/-1-1);
  \draw (426bp,218bp) node {$2$};
  \draw [black,->] (11000-1) ..controls (225.87bp,75.542bp) and (170.85bp,52.289bp)  .. (-1-11110-1-1);
  \draw (220bp,60bp) node {$0$};
  \draw [blue,->] (110-100) ..controls (530.1bp,73.867bp) and (595.85bp,46.072bp)  .. (11-2-200);
  \draw (590bp,60bp) node {$1$};
\end{tikzpicture}
\]
\caption{The top of the crystal $\RC(\Lambda_0)$ in type $A_2^{(1)}$, created using Sage.}\label{fig:graph}
\end{figure}
\end{ex}

\appendix
\section{Calculations using Sage}\label{sec:sage}

We begin by setting up the Sage environment to give a more concise printing.

\begin{lstlisting}
sage: RiggedConfigurations.global_options(display="horizontal")
\end{lstlisting}

We construct our the rigged configuration from Example~\ref{ex:runningrig} (in the $U_q^{\prime}(\g)$ setting).

\begin{lstlisting}
sage: RC = RiggedConfigurations(['D',5,1], [[1,2], [2,1], [3,1]])
sage: hw = RC(partition_list=[[2],[1,1],[1,1],[1],[1]]); hw
0[ ][ ]0   0[ ]0   1[ ]1   0[ ]0   0[ ]0
           0[ ]0   1[ ]1
sage: elt = hw.f_string([2,3,5,3,5,4,4,1,2,3,2,3,1]); elt
-1[ ][ ]-1   1[ ][ ]1   0[ ][ ]0    0[ ][ ]0   0[ ][ ]0
-1[ ][ ]-1   1[ ][ ]1   0[ ][ ]-2   0[ ]0      0[ ]0
             1[ ]1      0[ ]0
                        0[ ]0
\end{lstlisting}
Alternatively, one could construct $(\nu,J)$ from Example~\ref{ex:runningrig} directly by specifying the partitions and corresponding labels.
\begin{lstlisting}
sage: nu = RC(partition_list=[[2,2],[2,2,1],[2,2,1,1],[2,1],[2,1]],\
....: rigging_list=[[-1,-1],[1,1,1],[0,-2,0,0],[0,0],[0,0]])
\end{lstlisting}

The crystal $\RC(\infty)$ and $\RC(\lambda)$ has been implemented by the second author in Sage.  We conclude with examples.

\pagebreak
\begin{ex}
Let $\g_0 = D_5$. 
\begin{lstlisting}
sage: RC = crystals.infinity.RiggedConfigurations("D5")
sage: nu0 = RC.highest_weight_vector()
sage: elt = nu0.f_string([4,5,2,1,4,4,3,2,4,5,5,1,3]); elt
-2[ ][ ]-1   -2[ ]-1   2[ ][ ]-1   -6[ ][ ][ ][ ]-2   -4[ ][ ][ ]-1
             -2[ ]-1                                               
sage: elt.weight()
(2, 0, 0, 5, -1)
sage: [elt.epsilon(i) for i in RC.index_set()]
[1, 1, 1, 2, 1]
sage: [elt.phi(i) for i in RC.index_set()]
[-1, 1, 6, -4, -3]
\end{lstlisting}
\end{ex}

\begin{ex}
Let $\g_0 = E_7$.
\begin{lstlisting}
sage: RC = crystals.infinity.RiggedConfigurations(['E',7])
sage: nu0 = RC.highest_weight_vector()
sage: elt = nu0.f_string([1,3,4,2,5,6,7,4]); elt
-1[ ]0   0[ ]0   1[ ]1   -1[ ]-1   1[ ]1   0[ ]0   -1[ ]-1
                         -1[ ]-1                          
sage: elt.weight()
(1/2, -1/2, 1/2, -1/2, -1/2, 1/2, -1/2, 1/2)
sage: [elt.epsilon(i) for i in RC.index_set()]
[0, 0, 0, 1, 0, 0, 1]
sage: [elt.phi(i) for i in RC.index_set()]
[-1, 0, 1, 0, 1, 0, 0]
\end{lstlisting} 
\end{ex}

\begin{ex}
Let $\g = H_1^{(4)}$.
\begin{lstlisting}
sage: cm = CartanMatrix([
....: [2,-1,-1,-1],
....: [-1,2,-1,-1],
....: [-1,-1,2,-1],
....: [-1,-1,-1,2]])
sage: RC = crystals.infinity.RiggedConfigurations(cm)
sage: RC.index_set()
(0, 1, 2, 3)
sage: nu0 = RC.highest_weight_vector()
sage: elt = nu0.f_string([0,1,2,3,2,1,2,0,3,3,3,1,2]); elt
3[ ]4   0[ ]0   1[ ][ ][ ]2   1[ ][ ][ ]-1
3[ ]2   0[ ]0   3[ ]-1        3[ ]1
        0[ ]0
sage: elt.weight()
-7*Lambda[0] - 4*Lambda[1] - Lambda[2] - Lambda[3]
sage: [elt.epsilon(i) for i in RC.index_set()]
[0, 0, 1, 1]
sage: [elt.phi(i) for i in RC.index_set()]
[7, 4, 2, 2]
\end{lstlisting}
%We can compute the multiplicity of a root also.  Here, we compute the multiplicity of the root $\beta = -\alpha_0-\alpha_1-\alpha_2-\alpha_3$.
%\begin{lstlisting}
%sage: space = set([])
%sage: for p in Permutations([0,1,2,3]):
%....:     nu = nu0.f_string(p) 
%....:     if nu is not None:
%....:         space.add(nu)
%....:         
%sage: len(space)
%24
%\end{lstlisting}
%Thus $\operatorname{mult}(\beta) = 24$.
\end{ex}

\begin{ex}
Consider $\RC(\Lambda_0)$ with $\g = A_2^{(1)}$. The followings generates the crystal graph in Figure~\ref{fig:graph}:
\begin{lstlisting}
sage: P = RootSystem(['A',2,1]).weight_lattice()
sage: La = P.fundamental_weights()
sage: RC = crystals.RiggedConfigurations(['A',2,1],La[0])
sage: nu0 = RC.highest_weight_vector()
sage: nu0.f(0)
-1[ ]-1   (/)   (/)
sage: nu0.f_string([0,1])
0[ ]0   -1[ ]-1   (/)
sage: nu0.f_string([0,1,0])
sage: nu0.f_string([0,1,1])
sage: nu0.f_string([0,1,2])
1[ ]1   0[ ]0   0[ ]-1
sage: S = RC.subcrystal(max_depth=4)
sage: G = RC.digraph(subset=S)
sage: view(G, tightpage=True)
\end{lstlisting}
\end{ex}

\section*{Acknowledgements}

We would like to thank Anne Schilling for very valuable discussions and for reading a draft of this manuscript.  We would also like to thank Sara Billey, Ben Brubaker, Dan Bump, Gautam Chinta, Sol Friedberg, Dorian Goldfeld, Jeff Hoffstein, Anne Schilling, and Nicolas Thi\'ery for organizing the ICERM semester program entitled ``Automorphic Forms, Combinatorial Representation Theory, and Multiple Dirichlet series,'' where the idea for this project originated.  This work was also aided by Sage Mathematical Software \cite{combinat,sage}, in which the second named author designed packages corresponding to the work in this paper.  Finally, the authors would like to thank the anonymous referees for there helpful comments and insight.

\bibliography{RC}{}
\bibliographystyle{amsplain}
\end{document}